\documentclass[aop,MSNbibl,dvips]{arximspdf}
\usepackage{mathbh}
\usepackage{graphicx}

\doi{10.1214/12-AOP764} %kopijuoti is PTS
\volume{41}
\issue{5}
\pubyear{2013}
\firstpage{3542}
\lastpage{3581}

\makeatletter
\newcommand{\rrvert}{\vert}
\newcommand{\llvert}{\vert}

\newcommand{\underset}[3]{\mathop{#2}_{#1}^{#3}}

\newcommand{\E}{\mathbb E}
\newcommand{\R}{\mathbb{R}}
\newcommand{\N}{\mathbb{N}}
\newcommand{\C}{\mathbb{C}}
\newcommand{\Z}{\mathbb{Z}}
\renewcommand{\P}{\mathbb{P}}
\newcommand{\ZZZ}{\mathcal{Z}}
\newcommand{\RRR}{\mathcal{R}}

\newcommand{\supp}{\operatorname{supp}}
\newcommand{\sgn}{\operatorname{sgn}}

\newcommand{\MMM}{\mathfrak{M}}
\newcommand{\Mb}{\mathbb{M}}
\newcommand{\NNN}{\mathfrak{N}}
\newcommand{\KKK}{\mathfrak{K}}

\newcommand{\PPP}{\mathfrak{P}}
\newcommand{\QQQ}{\mathfrak{Q}}
\newcommand{\CCC}{\mathfrak{C}}

\newcommand{\eps}{\varepsilon}

\newcommand{\todistr}{\underset{n\to\infty}{\longrightarrow}{d}}
\newcommand{\toweak}{\underset{n\to\infty}{\longrightarrow}{w}}
\newcommand{\toprobab}{\underset{n\to\infty}{\longrightarrow}{P}}

\newcommand{\toas}{\underset{n\to\infty}{\longrightarrow}{\mathrm{a.s.}}}
\newcommand{\ton}{\underset{n\to\infty}{\longrightarrow}{}}

\newcommand{\bsl}{\setminus}
\newcommand{\ind}{\mathbh{1}}

\newtheorem{theorem}{Theorem}[section]
\newtheorem{lemma}[theorem]{Lemma}
\newtheorem{corollary}[theorem]{Corollary}
\newtheorem{proposition}[theorem]{Proposition}

\newproclaim{remark}[theorem]{Remark}

\makeatother

\begin{document}
\begin{frontmatter}

\title{Roots of random polynomials whose coefficients have logarithmic tails}
\runtitle{Roots of random polynomials}

\begin{aug}
\author[A]{\fnms{Zakhar} \snm{Kabluchko}\corref{}\ead[label=e1]{zakhar.kabluchko@uni-ulm.de}}
\and
\author[B]{\fnms{Dmitry} \snm{Zaporozhets}\ead[label=e2]{zap1979@gmail.com}}
\runauthor{Z. Kabluchko and D. Zaporozhets}
\affiliation{Ulm University and Russian Academy of Sciences}
\address[A]{Institute of Stochastics\\
Ulm University\\
Helmholtzstr. 18\\
89069 Ulm\\
Germany\\
\printead{e1}} %adresu isvedimo komanda gale!
\address[B]{St. Petersburg Branch \\
Steklov Institute of Mathematics\\
Fontanka Str. 27\\
191011 St. Petersburg\\
Russia\\
\printead{e2}}
\end{aug}

% HISTORY:
\received{\smonth{10} \syear{2011}}

% ABSTRACT
%
\begin{abstract}
It has been shown by Ibragimov and Zaporozhets [In \textit{Prokhorov
and Contemporary Probability Theory} (2013) Springer] that the complex
roots of a random polynomial $G_n(z)=\sum_{k=0}^n \xi_k z^k$ with
i.i.d. coefficients $\xi_0,\ldots,\xi_n$ concentrate a.s. near the unit
circle as $n\to\infty$ if and only if ${\E\log_{+}} |\xi_0|<\infty$. We
study the transition from concentration to deconcentration of roots by
considering coefficients with tails behaving like $L({\log}|t|)({\log}
|t|)^{-\alpha}$ as $t\to\infty$, where $\alpha\geq0$, and $L$ is a
slowly varying function. Under this assumption, the structure of
complex and real roots of $G_n$ is described in terms of the least
concave majorant\vspace*{1pt} of the Poisson point process on
$[0,1]\times (0,\infty)$ with intensity $\alpha v^{-(\alpha+1)}
\,du\,dv$.
\end{abstract}

% KEYWORDS
%

\end{frontmatter}

%s1 #&#
\section{Introduction and statement of results}\label{secintro}
%s1.1 #&#
\subsection{Introduction}\label{subsecintro}
Let $\xi_0,\xi_1,\ldots$ be i.i.d. nondegenerate random variables
with values in $\C$. Let $\ZZZ_n$ be the collection of complex roots
(counted according to their multiplicities) of the random polynomial
%
%e1 #&#
%
\begin{equation}
G_n(z)=\sum_{k=0}^n
\xi_k z^k.
\end{equation}
For $0\leq a\leq b$ denote by $R_n(a,b)$ the number of roots of $G_n$
in the ring $\{z\in\C\dvtx a\leq|z|\leq b\}$.
%The behavior of the zeros of $G_n$ as $n\to\infty$ has been studied by
%many authors starting with...
Improving on a result of {\v{S}}paro and {\v{S}}ur \cite{sparosur},
Ibragimov and Zaporozhets \cite{izlog} show that
%
%e2 #&#
%
\begin{equation}
\frac1n R_n(1-\eps,1+\eps) \toas1
\end{equation}
for every $\eps\in(0,1)$, if and only if
%
%e3 #&#
%
\begin{equation}
\label{eqlogmoment} {\E\log_+} |\xi_0|<\infty.
\end{equation}
Here, $\log_+ x=\max(\log x, 0)$. Without any assumptions on the
distribution of $\xi_0$, Ibragimov and Zaporozhets \cite{izlog} also
prove that for every $\alpha,\beta$ such that $0\leq\alpha<\beta
\leq2\pi$,
%
%e4 #&#
%
\begin{equation}
\frac1n\sum_{z\in\ZZZ_n} \ind_{\{\alpha\leq\arg z\leq\beta\}} \toas
\frac{\beta-\alpha}{2\pi}.
\end{equation}
Thus, under a very mild moment condition, the complex roots of $G_n$
concentrate near the unit circle uniformly by the argument as $n\to
\infty$.

Imposing additional conditions on the distribution of $\xi_0$ it is
possible to obtain more precise information about the asymptotic
concentration of the roots near the unit circle. In the case when $\xi
_0$ belongs to the domain of attraction of an $\alpha$-stable law,
$\alpha\in(0,2]$, Ibragimov and Zeitouni \cite{ibrzeit} show that
for every $t>0$,
%
%e5 #&#
%
\begin{equation}
\label{eqibrzeit} \lim_{n\to\infty}\frac1n \E R_n \biggl(1-
\frac{t}{n},1+\frac
{t}{n} \biggr)= \frac{1+e^{-\alpha t}}{1-e^{-\alpha t}}-\frac2{\alpha t}.
\end{equation}
This is a generalization of the result of Shepp and Vanderbei
\cite{sheppvanderbei} who consider real-valued standard Gaussian coefficients.

On the other hand, if ${\E\log_+} |\xi_0|=\infty$ and thus there is
no concentration near the unit circle, it is also possible to describe
the asymptotic behavior of the roots when the tail of $|\xi_0|$ is
extremely heavy. G{\"o}tze and Zaporozhets \cite
{goetzezaporozhets} prove that if the distribution of ${\log_+\log_+}|\xi
_0|$ has a slowly varying tail, then the complex roots of $G_n$
concentrate in probability on two circles centered at the origin whose
radii tend to zero and infinity, respectively. See also \cite{z,zn} for
more results in the case of extremely heavy tails.

Up to now, the behavior of the roots has been unknown when the tail of
$\xi_0$ is somewhere between the two cases described above. The aim of
this paper is to consider a class of distributions which in some sense
continuously links the above cases. We will consider coefficients with
logarithmic power-law tails. More precisely,
we make the following assumption: for some $\alpha\geq0$,
%
%e6 #&#
%
\begin{equation}
\label{eqtail1} %\P[|\xi_0|>t]\sim\frac{L(\log t)}{(\log t)^{\alpha}}, t
\quad \bar F(t):=\P\bigl[{\log}|\xi_0|>t\bigr]
\qquad\mbox{is regularly varying at } {+}\infty\mbox{ with index } {-}\alpha.
\end{equation}
This class of distributions includes distributions with both finite
($\alpha>1$) and infinite ($\alpha<1$) logarithmic moments. We will
obtain a precise information on how the concentration of the roots near
the unit circle becomes destroyed as $\alpha$ approaches~$1$ from above
and how the roots behave when there is no concentration ($\alpha<1$).

%As we will show, the distribution of roots becomes more and more
%organized as the parameter $\alpha$ (which can be thought of as a sort
%of temperature) decreases from $+\infty$ to $0$.
The case $\alpha=+\infty$ corresponds formally to the light or
power-law tails studied in \cite{sheppvanderbei,ibrzeit}.
The roots are concentrated near the unit circle and, apart from this,
no global organization is apparent. %in which case the distribution of
%roots is highly disorganized as can be judged from the central limit
%theorems of \cite{ibrmasl1,ibrmasl2}.
We will prove that as $\alpha$ becomes finite, the distribution of
roots becomes highly organized; see Figure~\ref{figroots}. The roots
``freeze'' on a random set of circles centered at the origin. Both the
radii of the circles and the distribution of the roots among the
circles are random; however, the distribution of the roots on each
circle is uniform by the argument. As long as $\alpha$ stays above~$1$,
the logarithmic moment is finite, and the circles approach the unit
circle at rate $n^{1/\alpha-1}$ (ignoring a slowly varying term), in
full agreement with the result of \cite{izlog}. Note also that for
$\alpha$ close to $+\infty$, this rate is close to the rate $1/n$
appearing in (\ref{eqibrzeit}). As $\alpha$ becomes equal to $1$, we
have a transition from finite to infinite logarithmic moment. We will
show that if $\bar F(t)\sim c/t$ as $t\to+\infty$, then the empirical
measure formed by the roots of $G_n$ converges weakly (without
normalization) to a random probability measure concentrated on an
infinite number of circles with random radii. For the first time, the
roots are not concentrated near the unit circle. As $\alpha$ becomes
smaller than $1$, the circles divide into two groups approaching $0$
and $\infty$ at the rates $\pm n^{1/\alpha -1}$, on the logarithmic
scale. The number of circles, which was infinite for $\alpha\geq1$,
becomes finite for $\alpha<1$ and decreases to $2$ as $\alpha\to0$. At
$\alpha=0$ the roots freeze on just $2$ circles located very close to
$0$ and $\infty$, in accordance with G\"otze and Zaporozhets
\cite{goetzezaporozhets}, whose results we will strengthen. At
$\alpha=0$, the empirical measure formed by the roots becomes almost
deterministic: the only parameter which remains random after taking the
limit $n\to\infty$ is the proportion of the roots close to $0$ (or to
$\infty$), which is uniform on $[0,1]$.

%%%%%%%%%%%%%%%%%%%%%FIGURE 1%%%%%%%%%%%%%%%%%%%%%%%%%%
%
%f1 #&#
%
\begin{figure}

\includegraphics{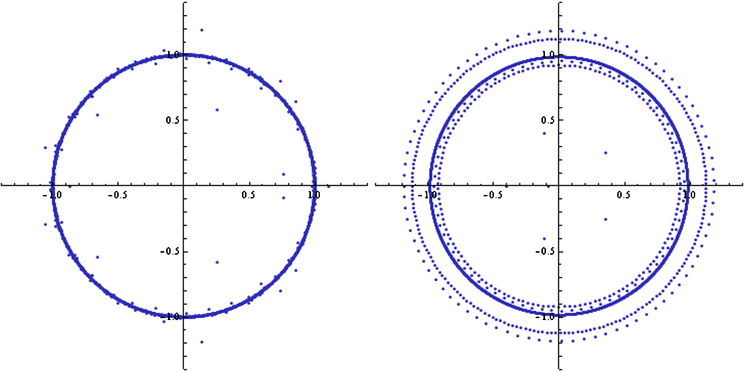}

\caption{Roots of a random polynomial of degree $n=2000$ whose (real)
coefficients are (left) standard normal, (right) such that $\P[\log\xi
_0>t]=1/t^{2}$ for $t\geq1$.}\label{figroots}
\end{figure}
%
%%%%%%%%%%%%%%%%%%%%%%%%%%%%%%%%%%%%%%%%%%%%%%%%%%%%%%%

An interesting phenomenon we will encounter is the appearance of the
long-range dependence between the roots under condition (\ref
{eqtail1}). Consider a random polynomial $G_n$ of high degree, and
suppose that we know that it has a root at some point $z_0\in\C$. In
the case of coefficients from the domain of attraction of a stable law,
this information has almost no influence on the other roots of $G_n$,
except for the roots located in an infinitesimal neighborhood of $z_0$.
However, for coefficients with logarithmic power-law tails, the
knowledge about the existence of a root at $z_0$ implies that there
exists (with high probability) a circle of roots containing $z_0$.
Moreover, the radii of the other circles of roots are influenced by the
existence of the root at $z_0$. We observe a long-range dependence
between the roots: the conditional distribution of roots, given that
there is a root at $z_0$, differs, even on the global scale, from the
unconditional distribution of roots.

If the random variables $\xi_i$ are real-valued, we will also analyze
the real roots of~$G_n$. %If the coefficients are in the domain of
%attraction of the $\beta$-stable law, $\beta\in(0,2]$, then it has
%been shown in \cite{ibrmasl1} that the expected number of real roots
%of $G_n$ is asymptotically $c(\beta)\log n$. As $\beta$ decreases from
%$2$ to $0$, the function $c(\beta)$ increases from $c(2)=2/\pi$ to
%$c(0)=1$.
For a particular family of distributions satisfying (\ref{eqtail1})
with $\alpha>1$, Shepp and Farahmand \cite{sheppfarahmand}
show that the expected number of real roots of $G_n$ is asymptotically
$c(\alpha)\log n$ with $c(\alpha)=\frac{2\alpha-2}{2\alpha-1}$. As
$\alpha$ decreases from $+\infty$
%(corresponding formally to the stable domain of attraction with $\beta
to $1$ the function $c(\alpha)$ decreases from $1$ to $0$.
%We will show that the number of real roots continues to decrease as $
We will complement this result by showing that for $\alpha\in(0,1)$,
the number of real roots of $G_{n}$ has two subsequential
distributional limits as $n\to\infty$ along the subsequence of
even/odd integers. This means that for $\alpha\in(0,1)$ the
polynomial $G_n$ has, roughly speaking, $O(1)$ real roots. Finally, we
will prove that for $\alpha=0$, the number of real roots of $G_n$ can
take asymptotically only the values $0,\ldots,4$ and compute the
probabilities of these values.

%Then we show that asymptotically the roots behave as follows. There
%exists a point process $\{R_i\}$ on $\R$ (finite for $\alpha<1$ and
%infinite for $\alpha\geqslant1$) such that the roots concentrate to
%the centered circles with radii $\exp(n^{1/\alpha-1}R_i)$,. In
%particular, for $\alpha<1$ the roots tend to infinity and origin at
%the rate order, for $\alpha>1$ the roots concentrate to the unit
%circle with the rate of order $n^{1/\alpha-1}$, and for $\alpha=1$ the
%roots have a non-degenerate limit distribution.

%s1.2 #&#
\subsection{Complex roots}\label{subseccomplexroots}
Given a complex number $z=|z|e^{i \arg z}$ and $a\in\R$, we write
\[
z^{\langle a\rangle}=|z|^{a}e^{i\arg z}.
\]
%
%It will be convenient to consider the moduli of the roots on the
%logarithmic scale. Define %
The next theorem describes the structure of complex roots of $G_n$. Let
$\delta(z)$ be the unit point mass at $z$. Denote by $\bar\C=\C\cup
\{\infty\}$ the Riemann sphere.
%and by $\MM(\bar\C)$ the space of probability measures on $\bar\C$
%endowed with the topology of weak convergence.
We need normalizing sequences $a_n,b_n$ such that
%e7 #&#
%
\begin{equation}
\label{eqdefan} \bar F(a_n)\sim\frac1n \qquad\mbox{as } n\to\infty,
\qquad b_n=\frac n{a_n}.
\end{equation}

%th1.1 #&#
%
\begin{theorem}\label{theocomplex}
If the tail condition (\ref{eqtail1}) is satisfied with some $\alpha
>0$, then we have the following weak convergence of random probability
measures on~$\bar\C$:
%$\MM(\bar\C)$-valued random elements:
%
\[
\frac1 n\sum_{z\in\ZZZ_n} \delta\bigl({z^{\langle b_n \rangle
}}
\bigr)\toweak\Pi_{\alpha}.
\]
The limiting random probability measure $\Pi_{\alpha}$ is a.s. a
convex combination of at most countably many uniform measures
concentrated on circles centered at the origin.
\end{theorem}
%
%re1.2 #&#
%
\begin{remark}[(On convergence of random measures)]
Let $E$ be a locally compact metric space. Denote by $\Mb(E)$ the
space of locally\vadjust{\goodbreak} finite Borel measures on~$E$. Endowed with the
topology of vague convergence, $\Mb(E)$ becomes a Polish space;
see \cite{resnickbook}, Section 3.4. A \textit{random measure} on
$E$ is a random element with values in $\Mb(E)$. A sequence $\mu_n$
of random measures \textit{converges weakly} to a random measure $\mu
$ if $\lim_{n\to\infty}\E F(\mu_n)=\E F(\mu)$ for every
continuous, bounded function $F\dvtx\Mb(E)\to\R$. Equivalently, $\int_E
f \,d\mu_n$ converges in distribution to $\int_E f\,d\mu$ for every
compactly supported continuous function $f\dvtx E\to\R$; see
\cite{resnickbook}, Section 3.5.
\end{remark}
For $\alpha\geq1$ the logarithmic moment condition (\ref
{eqlogmoment}) is satisfied, which by \cite{izlog} means that the
roots should concentrate near the unit circle. In the next corollary we
compute the rate of convergence of the roots to the unit circle.

%co1.3 #&#
%
\begin{corollary}
Let $\alpha\geq1$. As $n\to\infty$, the random probability measure
\[
\frac1 n\sum_{z\in\ZZZ_n} \delta\bigl(b_n\bigl(|z|-1\bigr)
\bigr)
\]
converges weakly to a random, a.s. purely atomic probability measure
on $\R$.
\end{corollary}
In the case $\bar F(t)\sim c/t$ as $t\to+\infty$, where $c>0$, the
logarithmic moment condition (\ref{eqlogmoment}) just fails. We have
no concentration of the roots near the unit circle for the first time.
In this case, Theorem~\ref{theocomplex} simplifies as follows.
%
%co1.4 #&#
%
\begin{corollary}
Suppose that $\bar F(t)\sim c/t$ as $t\to+\infty$. Then, the
empirical measure $\frac1n \sum_{z\in\ZZZ_n} \delta(z)$ converges
weakly to some nontrivial limiting random probability measure on $\C$.
\end{corollary}

% It will turn out that $\nu_{\alpha}$ is a weighted sum of lengths
%measures on circles at the origin. There are two sources of
%randomness: the weights of the circles and their radii. For $\alpha\in
%(0,1)$ the number of circles is finite. For $\alpha=1$ the number of
%circles is countably infinite and we have exactly two concentration
%points $0$ and $\infty$.
We proceed to the description of the random probability measure $\Pi
_{\alpha}$. Let $\rho=\sum_{k=1}^{\infty}\delta(U_k,V_k)$ be a
Poisson point process on $[0,1]\times(0,\infty)$ with intensity
measure $\alpha v^{-(\alpha+1)} \,du \,dv$. Equivalently, $U_k$, $k\in\N
$, are i.i.d. random variables with a uniform distribution on $[0,1]$
and, independently, $V_k=W_k^{-1/\alpha}$, where $W_1,W_2,\ldots$ are
the arrival times of a homogeneous Poisson point process on $(0,\infty
)$ with intensity $1$.
%Alternatively, $\sum_{k=1}^{\infty}\delta_{Z_k}$ is a Poisson point
%process on $(0,\infty)$ with intensity
%$c\alpha\,dz/z^{\alpha+1}$ and independently, $U_k$,$k\in\N$ are iid
%random variables distributed uniformly on $[0,1]$.
Of major importance for the sequel is the \textit{least concave
majorant} (called simply \textit{majorant}) of $\rho$ (see
Figure~\ref{figmajorant}) which is a function $\CCC_{\rho}\dvtx[0,1]\to
[0,\infty)$ defined by
\[
\CCC_{\rho}(t):=\inf_{f} f(t),\qquad t\in[0,1],
\]
where the infimum is taken over the set of all concave functions
$f\dvtx[0,1]\to[0,\infty)$ satisfying $f(U_k)\geq V_k$ for all $k\in\N$.
From a constructive viewpoint, the least concave majorant $\CCC_{\rho
}$ may be defined as follows. Let $(X_0,Y_0)$ be the a.s. unique atom
of $\rho$ having a maximal second coordinate $Y_0$ among all atoms of
$\rho$. Consider a horizontal line passing through $(X_0,Y_0)$. Rotate
this line around $(X_0,Y_0)$ in a \textit{clock-wise} direction until
it hits some atom of $\rho$, denoted by $(X_1,Y_1)$, other than
$(X_0,Y_0)$. Continue to rotate the line in the clock-wise direction,
%%%%%%%%%%%%%%%%%%%%%%%%%%%%%FIGURE 2%%%%%%%%%%%%%%%%%%%%%%%%%%%%%%%
%
%f2 #&#
%
\begin{figure}

\includegraphics{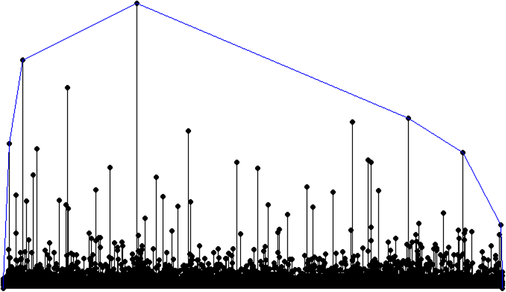}

\caption{The least concave majorant of the Poisson point process $\rho
$ on $[0,1]\times(0,\infty)$ with intensity $\alpha v^{-(\alpha+1)}
\,du \,dv$, where $\alpha=2$.}
\label{figmajorant}
\end{figure}
%
%%%%%%%%%%%%%%%%%%%%%%%%%%%%%%%%%%%%%%%%%%%%%%%%%%%%%%%%%%%%%%%%%%%%
this time around $(X_1,Y_1)$, until it hits some atom of $\rho$,
denoted by $(X_2,Y_2)$, other than $(X_1,Y_1)$. Continue to rotate the
line around $(X_2,Y_2)$, and so on. The procedure is terminated if at
some time the line hits the point $(1,0)$. [As we will see later, this
happens a.s. if and only if $\alpha\in(0,1)$.] Otherwise, the
procedure is repeated indefinitely. Analogously, we can start with a
horizontal line passing through $(X_0,Y_0)$ and rotate it in an \textit
{anti-clockwise} direction obtaining a sequence of points
$(X_{-1},Y_{-1}), (X_{-2},Y_{-2}),\ldots\,$. The sequence may eventually
terminate at $(0,0)$. [We will see that this happens a.s. if and only
if $\alpha\in(0,1)$.] Now, join any point $(X_k,Y_k)$ to the next
point $(X_{k+1},Y_{k+1})$ by a line segment. The polygonal path
constructed in this way is the graph of the majorant $\CCC_{\rho}$.
The points $(X_k,Y_k)$ are called the \textit{vertices} of the
majorant, the intervals $[X_k,X_{k+1}]$ are called the \textit
{linearity intervals} of the majorant. The least concave majorant $\CCC
_{\rho}$ is thus a piecewise linear function with at most countably
many linearity intervals. We write $\CCC_{\rho}$ in the form
%
%e8 #&#
%
\begin{equation}
\label{eqCCCrho} \CCC_{\rho}(t)=S_k-R_k t,\qquad
t\in[X_k,X_{k+1}].
\end{equation}
%
%The part of the majorant located to the right of $(X_0,Y_0)$ is the
%polygonal line joining the points $(X_0,Y_0),(X_1,Y_1),\ldots\,$.
%Starting with the horizontal line passing through $(X_0,Y_0)$ and
%rotating it in an anti-clockwise direction as long as we do not hit
%$(0,0)$ we obtain the left part of the majorant which is a polygonal
%line joining the points $(X_0,Y_0), (X_{-1},Y_{-1}),\ldots\,$.
%Thus, the least concave majorant is a piecewise linear function with
%at most countably many intervals of linearity. As we shall see later,
%the number of linearity intervals is finite iff $\alpha<1$.

The limiting random probability measure $\Pi_{\alpha}$ in Theorem
\ref{theocomplex} can be constructed as follows. For $r>0$ let
$\Lambda_r$ be the length measure (normalized to have total mass~$1$)
on the circle $\{z\in\C\dvtx|z|=r\}$. Then
\[
\Pi_{\alpha}=\sum_{k} (X_{k+1}-X_k)
\Lambda_{\exp(R_k)},
\]
where the (finite or infinite) sum is taken over all linearity
intervals $[X_k,X_{k+1}]$ of the majorant $\CCC_{\rho}$.
Thus Theorem~\ref{theocomplex} states that the roots of $G_n$
asymptotically concentrate on random circles which correspond to the
linearity intervals of the majorant. The radii of these random circles
are $\exp({R_k})$, where the $R_k$'s are the negatives of the slopes
of the majorant. The proportion of roots on any circle is the length of
the corresponding linearity interval.

Our next result describes the distribution of the complex roots of
$G_n$ in the case $\alpha=0$.
We assume that
%
%e9 #&#
%
\begin{equation}
\label{eqtailalpha0} \bar F(t):= \P\bigl[{\log}|\xi_0|>t\bigr]
\qquad\mbox{is slowly
varying at } {+}\infty.
\end{equation}
%
%In the next theorem we show that under \eqref{eqtailalpha0} the
%complex roots of $G_n$ are located on two circles, one of them close
%to $0$ and the other one close to $\infty$.
We will show that under (\ref{eqtailalpha0}), with probability
close to $1$, the complex roots of $G_n$ are located on just $2$
concentric circles, one of them with a radius close to $0$ and the
other one with a radius close to $\infty$.
A weaker result was obtained by G{\"o}tze and Zaporozhets \cite
{goetzezaporozhets} under a more restrictive assumption on the tails.
Let $\tau_n$ be the index of the maximal (in the sense of absolute
value) coefficient of $G_n$, that is, $\tau_n\in\{0,\ldots,n\}$ is
such that $|\xi_{\tau_n}|=\max_{k=0,\ldots,n}|\xi_k|$.
%Denote by $w_{kn}$,$k=1,\ldots,\tau_n$ the roots of the equation $z^{
Denote by $w_{1n},\ldots,w_{\tau_n n}$ the roots of the equation $\xi
_{\tau_n} z^{\tau_n}+\xi_0=0$ and by $w_{(\tau_n+1)n},\ldots,w_{nn}$
the roots of the equation $\xi_n z^{n-\tau_n}+\xi_{\tau_n}=0$.
%
%th1.5 #&#
%
\begin{theorem}\label{theocomplexalpha0}
Suppose that (\ref{eqtailalpha0}) is satisfied and $\xi_0\neq0$
a.s. Fix some $A>0$. Then the probability that the following three
statements hold simultaneously goes to $1$ as $n\to\infty$:
\begin{longlist}[(3)]
\item[(1)] $\tau_n$ is uniquely defined;
\item[(2)] it is possible to renumber the roots $z_{1n},\ldots,z_{nn}$
of $G_n$ so that
\[
|z_{kn}-w_{kn}|<e^{-n^A} |w_{kn}|,\qquad 1
\leq k\leq n;
\]
\item[(3)] we have $|w_{kn}|<e^{-n^A}$ for $1\leq k\leq\tau_n$ and
$|w_{kn}|>e^{n^A}$ for $\tau_n<k\leq n$.
\end{longlist}
\end{theorem}
%
%co1.6 #&#
%
\begin{corollary}\label{corcomplexalpha0}
Under (\ref{eqtailalpha0}), the empirical measure $\frac1n \sum_{z\in
\ZZZ_n} \delta(z)$ converges weakly, as a random probability
measure on the Riemann sphere $\bar\C$, to
$U\delta(0)+(1-U)\delta(\infty)$, where $U$ is a random variable
with a uniform distribution on $[0,1]$.
\end{corollary}

%s1.3 #&#
\subsection{Properties of the majorant}\label{subsecpropmaj}
In this section we study some of the properties of the least concave
majorant $\CCC_{\rho}$. Note that random convex hulls similar to
$\CCC_{\rho}$ appeared in the literature; see \cite{majumdar} and
the references therein. The next proposition will be used frequently.
%The least concave majorant of the Poisson process $\rho$ is an
%interest object on its own.
%
%pr1.7 #&#
%
\begin{proposition}\label{propnumberverticesmajorant}
Let $L_{\alpha}$ be the number of linearity intervals of the majorant
$\CCC_{\rho}$.
If $\alpha\in(0,1)$, then $L_{\alpha}<\infty$ a.s.
If $\alpha\geq1$, then $L_{\alpha}=\infty$ a.s. Moreover, in this
case any neighborhood of $0$ (as well as any neighborhood of $1$)
contains infinitely many linearity intervals of $\CCC_{\rho}$ a.s.,
and we have $\lim_{k\to-\infty}R_k=-\infty$ and $\lim_{k\to
+\infty}R_k=+\infty$ a.s.
\end{proposition}
\begin{pf}
Take any $\eps>0$ and consider the set $D_{\eps}$ of all pairs
$(x,y)\in[0,1]\times(0,\infty)$ such that $y>\eps x$. Integrating\vadjust{\goodbreak}
the intensity of $\rho$ over $D_{\eps}$ we see that $\rho(D_{\eps
})=\infty$ a.s. if and only if $\alpha\geq1$.
If $\alpha\in(0,1)$, we have only finitely many points above any line
$y>\eps x$ and hence, the majorant $\CCC_{\rho}$ has a well-defined
first segment starting at $(0,0)$.
On the other hand, if $\alpha\geq1$, then no such first segment
exists and consequently, we have infinitely many linearity intervals of
$\rho$ in any neighborhood of $0$. By symmetry, the same is true for
the point $1$.
\end{pf}
The distribution of $L_{\alpha}$ in the case $\alpha\in(0,1)$ seems
difficult to characterize. In the next theorem, we compute the
expectation of $L_{\alpha}$ in terms of the modular constant $C(\beta
)$ introduced by Barnes \cite{barnes} in his theory of
the double Gamma function. Let $\psi(z)=\Gamma'(z)/\Gamma(z)$ be the
logarithmic derivative of the Gamma function. Barnes
\cite{barnes} showed that the following limit exists for $\beta>0$:
%
%e10 #&#
%
\begin{equation}
\label{eqbarnesmodularconst} C(\beta):=\lim_{n\to\infty} \Biggl\{\sum
_{m=1}^{n} \psi(m\beta)- \biggl(n+\frac12-\frac1{2
\beta} \biggr) \log(n\beta)+n \Biggr\}. %\frac{\Gamma'(m\beta)}{\Gamma(m
\end{equation}
The role of the constant $C(\beta)$ in the theory of the double Gamma
function is similar to the role of the Euler--Mascheroni constant
$\gamma=\lim_{n\to\infty}(\sum_{k=1}^n\frac1k-\log n)$ in the
theory of the usual Gamma function. %$\gamma=\lim_{N\to\infty}(
%
%th1.8 #&#
%
\begin{theorem}\label{theointensity}
%Let $L_{\alpha}$ be the number of linearity intervals of the least
%concave majorant of the Poisson point process $\rho$.
For $\alpha\in(0,1)$, $\alpha\neq1/2$, we have
%(0,1)$, then $L_{\alpha}$ is finite a.s.
%
%e11 #&#
%
\begin{equation}
\label{eqtheoexpectnumbervertices} \E L_{\alpha} = 2+\frac{2-2\alpha
}{2\alpha-1}
\biggl(1-2C(1-\alpha)+\frac{\log
(1-\alpha)-\alpha\gamma}{1-\alpha} \biggr).
\end{equation}
For $\alpha=1/2$ the result should be interpreted by continuity.
\end{theorem}
We will provide a representation of $\E L_{\alpha}$ as a definite
integral in equation (\ref{eqexpLalphaint}) below. Using this
representation it is possible to compute the value of $\E L_{\alpha}$
in closed form for any rational $\alpha$. Here are some
examples:\vspace*{-8pt}

\begin{table}[h]
\begin{tabular}{@{}lccccccc@{}}
\hline
$\alpha$ & $0$ & $1/4$ & $1/3$ & $1/2$ & $2/3$ & $3/4$ & $1$ \\
$\E L_{\alpha}$ & 2 & $ (\frac32-\frac{4}{3\sqrt3} )\pi
$ & $\frac{4\pi}{3\sqrt3}$ &$\frac32 +\frac{\pi^2}{8}$
& $2+\frac{2\pi}{3\sqrt3}$ &
$2+\frac{\pi}{2}$ & $+\infty$ \\
\hline
\end{tabular}
\end{table}\vspace*{-8pt}

The values at $\alpha=0$ and $\alpha=1$ should be understood as
one-sided limits. As a corollary, we have $L_\alpha\to2$ in
distribution as $\alpha\downarrow0$. Another way to see this is the
following theorem.
%show that $\lim_{\alpha\uparrow1}\E L_{\alpha}=+\infty$ which is in
%accordance with the fact that the number of segments of $\CCC_{\rho}$
%becomes infinite at $\alpha=1$. Since $C(1)=1/2$ by \cite{barnes}, we
%have $\lim_{\alpha\downarrow0}\E L_{\alpha}=2$.
%
%th1.9 #&#
%
\begin{theorem}\label{theoprobabtwoseg}
For $\alpha\in(0,1)$ we have $\P[L_{\alpha}=2]=1-\alpha$.
\end{theorem}

%s1.4 #&#
\subsection{Real roots}\label{subsecrealroots}
Suppose now that the coefficients of the polynomial $G_n(z)=\sum
_{k=0}^n \xi_k z^k$ are
i.i.d. \textit{real-valued} random variables.
Denote by $\RRR_n$ the collection of real roots of $G_n$, and let
$N_n=|\RRR_n|$ be the number of real roots.\vadjust{\goodbreak}
For a special family of distributions satisfying (\ref{eqtail1}) with
$\alpha>1$ Shepp and Farahmand \cite{sheppfarahmand}
showed that $\E N_n\sim\frac{2\alpha-2}{2\alpha-1}\log n$ as $n\to
\infty$. In the next theorem we describe the positions of the real
roots of $G_n$ in the limit $n\to\infty$ for every $\alpha>0$.
%A point process on a locally compact space $M$ is a random element
%with values in the space of locally finite counting measures on $M$
%endowed with the topology of vague convergence.
Recall the notation $z^{\langle a\rangle}=|z|^a \sgn(z)$, where
$z,a\in\R$. Define a point process $\Upsilon_n$ on $\R$ by
\[
\Upsilon_n=\sum_{z\in\RRR_n} \delta
\bigl(z^{\langle b_n \rangle}\bigr).
\]
Recall that a random measure $\mu$ is called a point process if the
random variable $\mu(K)$ is integer-valued for every compact set $K$;
see \cite{resnickbook}, Section~3.1. In addition to (\ref{eqtail1})
we assume that the following limit exists:
%
%e12 #&#
%
\begin{equation}
\label{eqtailreal} c:=\lim_{t\to+\infty}\frac{\P[\xi_0>t]}{\P[|\xi
_0|>t]}\in[0,1].
\end{equation}

%th1.10 #&#
%
\begin{theorem}\label{theoreal}
Suppose that (\ref{eqtail1}) and (\ref{eqtailreal}) hold with some
$\alpha>0$. Write $p=\P[\xi_0>0]$ and suppose that $\xi_0\neq0$ a.s.
\begin{longlist}[(2)]
\item[(1)] For $\alpha\geq1$ the point process $\Upsilon_{n}$ converges
weakly to some point process $\Upsilon_{\alpha,c}$ on $\R\bsl\{0\}$.
\item[(2)] For $\alpha\in(0,1)$ the point process $\Upsilon_{2n}$
(resp., $\Upsilon_{2n+1}$) converges weakly to some point
process $\Upsilon_{\alpha,c,p}^{+}$ (resp.,
$\Upsilon_{\alpha,c,p}^{-}$) on $[-\infty,+\infty]$ and on $\R$.
\end{longlist}
\end{theorem}
The somewhat technical description of the point processes $\Upsilon
_{\alpha,c}$, $\Upsilon_{\alpha,c,p}^{\pm}$ is postponed to
Section~\ref{secproofrealdefproc}. Recall that by Theorem
\ref{theocomplex} the complex roots of $G_n$ are located asymptotically on
a set of random circles. Each circle crosses the real line at $2$
points. We will show that any of these points may or may not be a real
root of $G_n$ with some probabilities. For $\alpha\in(0,1)$ the point
processes $\Upsilon_{\alpha,c,p}^{\pm}$ have a.s. finitely many
atoms, whereas for $\alpha\geq1$ the atoms of the point process
$\Upsilon_{\alpha,c}$ accumulate a.s. at $\pm0$ and $\pm\infty$.
(Of course, this is related to Proposition
\ref{propnumberverticesmajorant}.) Since the map assigning to a finite
counting measure on $[-\infty,\infty]$ its total mass is continuous
(locally constant) in the weak topology, we obtain the following
statement on the number of real roots of $G_n$.
%
%co1.11 #&#
%
\begin{corollary}\label{correalnumber}
Suppose that (\ref{eqtail1}) and (\ref{eqtailreal}) hold with
$\alpha\in(0,1)$ and let $\xi_0\neq0$ a.s.
Then the sequence $\{N_{2n}\}_{n\in\N}$ (resp., $\{N_{2n+1}\}_{n\in
\N}$) converges in distribution to a random variable
$N^{+}_{\alpha,c,p}$ (resp., $N^{-}_{\alpha,c,p}$).
\end{corollary}
%
%re1.12 #&#
%
\begin{remark}
The expectations $\E N^{+}_{\alpha,c,p}$ and $\E N^{-}_{\alpha,c,p}$
can be computed using the representation of the point processes
$\Upsilon_{\alpha,c,p}^{\pm}$ given in Section
\ref{secproofrealdefproc}.
\[
\E N^{+}_{\alpha,c,p}= \E N^{-}_{\alpha,c,p}=
\bigl(2c(1-c)+\tfrac12 \bigr) (\E L_{\alpha}-2)+2(p+c-2pc)+1.
\]
For instance, if the distribution of $\xi_0$ is symmetric with respect
to the origin, then both expectation are equal to $\E L_{\alpha}$. We
conjecture that the convergence in Corollary~\ref{correalnumber}
holds in the $L^1$-sense.\vadjust{\goodbreak}
\end{remark}
%
%re1.13 #&#
%
\begin{remark}
The behavior of $N_n$ in the case $\alpha=1$ remains open. For $\alpha
=1$ the result of \cite{sheppfarahmand} turns formally into $\E
N_n=o(\log n)$, whereas the fact that $\Upsilon_{1,c}$ has infinitely
many atoms a.s. suggests that $\E N_n$ should be infinite. It is
natural to conjecture that for $\alpha=1$, we should have $\E N_n\sim
K\log\log n$ for some $K>0$.
\end{remark}
%
%We conjecture that $K=C'(1)=2$.

Finally, we investigate the number of real roots of $G_n$ in the case
$\alpha=0$.
%
%th1.14 #&#
%
\begin{theorem}\label{theorealalpha0}
Suppose that (\ref{eqtailalpha0}) and (\ref{eqtailreal}) hold,
$\xi_0\neq0$ a.s., and write $p=\P[\xi_0>0]$. Then, the sequence $\{
N_{2n}\}_{n\in\N}$ converges weakly to a random variable
$N^{+}_{0,c,p}$ and the sequence $\{N_{2n+1}\}_{n\in\N}$ converges
weakly to a random variable $N^{-}_{0,c,p}$ such that
%
%e13 #&#
%e14 #&#
%
\begin{eqnarray}
\label{eqtheorealalpha0eq1} \P\bigl[N^{+}_{0,c,p}=m\bigr]&=&
\cases{ \frac12 \bigl(cp^2+(1-c) (1-p)^2\bigr), &\quad$m=0$,
\vspace*{2pt}\cr
\frac12+ p(1-p), &\quad$m=2$,
\vspace*{2pt}\cr
\frac12 \bigl(c(1-p)^2+(1-c)p^2
\bigr), &\quad$m=4$;}
\\
\label{eqtheorealalpha0eq2} \P\bigl[N^{-}_{0,c,p}=m\bigr]&=&
\cases{ %cp^2+(1-c)(1-p)^2+p(1-p), &m=1,\\
%c(1-p)^2+(1-c)p^2+p(1-p), &m=3.
1-p-c+2p c, &\quad$m=1$,
\cr
p+c-2p c, &\quad$m=3$.}
\end{eqnarray}
\end{theorem}
%
%re1.15 #&#
%
\begin{remark}
If the distribution of $\xi_0$ is symmetric with respect to the
origin, we obtain the following results: $N^{+}_{0,1/2,1/2}$ takes the
values $0,2,4$ with probabilities $ 1/8, 3/4, 1/8$, and
$N^{-}_{0,1/2,1/2}$ takes the values $1,3$ with probabilities $1/2, 1/2$.
\end{remark}
%
%It is interesting to compare this with the result of Zaporozhets and
%Nazarov
%
%re1.16 #&#
%
\begin{remark}
For fixed $p\in[0,1]$, both
\[
\min_{c\in[0,1]} \E N^{+}_{0,c,p}\quad \mbox{and}\quad
\min_{c\in[0,1]} \E N^{-}_{0,c,p}
\]
are equal to $1+2\min(p,1-p)$.
The same number appeared in \cite{zn} as the minimal expected number
of real roots of a random polynomial.
\end{remark}

%where the minimum is taken over all $c_+,c_-\geq0$, $c_++c_-=1$.
%Thus, the distribution for which the

%s1.5 #&#
\subsection{Emergence of the majorant}
%We are going to describe a brief heuristic argument leading to our
%main results.
The least concave majorant which we encountered above is reminiscent of
the Newton polygons appearing when solving polynomial equations with
non-Archimedian (e.g., $p$-adic) coefficients; see~\cite{koblitzbook}, Chapter IV. Of course, our random polynomial $G_n$ has
complex (Archimedian) coefficients. %but we have the following
%non-Archimedian relation, valid for $x,y\in\R$ and large $n$:
%$$
%e^{nx}+e^{ny}\approx e^{n\max(x,y)}.
%$$
However, non-Archimedian effects will appear in the following way.
Consider the sum $c_1e^{nx_1}+\cdots+c_d e^{nx_d}$, where $x_i>0$ and
$c_i\in\C$. If $n$ is large, then the most easy way such sum may
become zero is if two terms, say $c_ke^{nx_k}$ and $c_le^{nx_l}$,
cancel each other and the other terms are much smaller than these two.
We will show that under (\ref{eqtail1}) similar considerations apply
to the polynomial $G_n(z)=\sum_{j=0}^n \xi_jz^j$ with high
probability: $z\in\C$ is a root of $G_n$ essentially only if two of
the terms, $\xi_kz^k$ and $\xi_lz^l$, cancel each other, and all
other terms are of smaller order. Geometrically, this means that the
points $(k,{\log}|\xi_k|)$ and $(l,{\log}|\xi_l|)$ are neighboring
vertices of the least concave majorant of the set $\{(j,{\log}|\xi
_j|)\dvtx j=0,\ldots,n\}$. The
nonzero roots of $\xi_kz^k+\xi_lz^l=0$ form a
regular polygon inscribed into the circle whose radius is the
exponential of minus the slope of the line joining the points $(k,{\log}
|\xi_k|)$ and $(l,{\log}|\xi_l|)$. Taking the union of such circles
over all segments of the majorant we obtain essentially all the roots
of~$G_n$. To complete the argument, we need to find the limiting form
of the majorant as $n\to\infty$. This is done using the following
proposition which is known in the extreme-value theory; see
\cite{resnickbook}, Corollary 4.19(ii). %or~\cite{resnickheavybook}, .
%%Note that condition~\eqref{eqtail1} can be rephrased by saying that $
%{\log}|\xi_0|$ belongs to the max-domain of attraction of the Fr\'echet
%distribution with parameter $\alpha$.
%
%pr1.17 #&#
%
\begin{proposition}\label{propresnick}
Let $\xi_0,\xi_1,\ldots$ be i.i.d. random variables
satisfying~(\ref{eqtail1}). Then the following convergence holds
weakly on the space of locally finite counting measures on $[0,1]\times
(0,\infty]$:
\[
\rho_n:=\sum_{k=0}^n \delta
\biggl(\frac kn, \frac{{\log}|\xi_k|}{a_n} \biggr) \toweak\sum
_{i=1}^{\infty}\delta(U_i,V_i)
=: \rho.
\]
Here, $\rho$ is a Poisson point process on $[0,1]\times(0,\infty)$
with intensity $\alpha v^{-(\alpha+1)}\,du\,dv$. We agree that the points
for which ${\log}|\xi_k|\leq0$ are not counted in $\rho_n$.
%Let also $\mu$ be a Poisson point process on $[0,1]\times\R_+$ with
%intensity $du \,dv/v^{\alpha+1}$.
%Then, $\mu_n$ converges to $\mu$ weakly.
\end{proposition}

The paper of Shepp and Farahmand \cite{sheppfarahmand}
seems to be the only work where random polynomials with coefficients
satisfying (\ref{eqtail1}) have been considered. The method used
there (characteristic functions) is very different from our approach
based on majorants. Whether the results of \cite{sheppfarahmand} can
be recovered (or strengthened) using our approach remains open. Let us
also mention that the least concave majorant appeared in the theory of
entire functions; see \cite{valironbook}, page 28. For example,
Hardy \cite{hardy} showed that the zeros of the
deterministic entire function $\sum_{k=0}^{\infty} z^{k^3}/(k^3)!$
have a circle structure similar to the structure of zeros of $G_n$
under (\ref{eqtail1}). Eigenvalues of random matrices with i.i.d.
heavy-tailed entries have been studied in \cite{bordenaveetal}.

%s2 #&#
\section{The main lemma}\label{secmainlemma}
The next lemma is the key step in the proof.
%Let $\eta>0$ be fixed. For $\alpha\geq1$ we denote by $I_n^*$ the set
%of all $i\in I_n$ such that the interval $[V_{i,n},W_{i,n}]$ has a
%non-empty intersection with $[\eta,1-\eta]$. For $\alpha\in(0,1)$ we
%take $I_n^*=I_n$.
Let $g(z)=\sum_{j=0}^n a_jz^j$ be a (deterministic) polynomial with
complex coefficients. Suppose that the points $(k,{\log}|a_k|)$ and
$(l,{\log}|a_l|)$, where $0\leq k<l\leq n$, are neighboring vertices on
the least concave majorant of the set $\{(j,{\log}|a_j|)\dvtx j=0,\ldots
,n\}
$. That is to say, for some $s,r\in\R$, we have
%
%e15 #&#
%
\begin{eqnarray}
\label{eqmodulusxikl} {\log}|a_{k}|&=&s-kr,\qquad
{\log}|a_{l}|=s-lr,\nonumber\\[-8pt]\\[-8pt]
h:\!&=&\mathop{\min_{0\leq j\leq n}}_{j\neq k,l}
\bigl(s-jr-{\log}|a_j|\bigr)>0.\nonumber
\end{eqnarray}
Here, we have assumed that no three points of the majorant are on the
same line. Note that $h$ measures the gap between the line passing
through the points $(k,{\log}|a_k|)$, $(l,{\log}|a_l|)$ and the points
lying below this line.
%
%le2.1 #&#
%
\begin{lemma}\label{lemmain}
If $\delta>0$ is such that $ne^{\delta n-h}<1-e^{-\delta}$, then in
the ring $e^{r-\delta}<|z|<e^{r+\delta}$ there are exactly $l-k$
roots of $g$. Moreover, if $\zeta$ is such that $2ne^{2\delta
n-h}<\zeta<\frac{\pi}{l-k}$, then the set
%
%e16 #&#
%
\begin{equation}
\label{eqsetsectorring} %Z(m,\delta, \zeta):=
\biggl\{ z\in\C\dvtx e^{r-\delta}<|z|<e^{r+\delta},
\biggl\llvert\arg z-\frac{\varphi+ 2\pi m}{l-k}\biggr\rrvert\leq\zeta
\biggr\},
\end{equation}
where $\varphi=\arg(-a_k/a_l)$, contains exactly one root of $g$ for
every $m=1,\ldots,\break l-k$.
\end{lemma}
Here, we agree to understand the distance between the arguments of
complex numbers as the geodesic distance on the unit circle. Also, let
the index $j$ be always restricted to $0\leq j\leq n$.

%Before porving Lemma~\ref{lemmain} we formulate one axillary result
%which is due to Ostrowski.
%Let $B$ be a closed region in the complex plane,
% the boundary of which consists of a finite number of regular arcs;
%let the functions $f(z), h(z)$ be regular on $B$.
% Assume that for all values of the real parameter $t$, running in the
%interval %$a\leqslant t\leqslant b$, the function $f(z)+t\cdot h(z)$
%is nonzero on the boundary of %$B$. Then the number of the roots of
%$f(z)+t\cdot h(z)$ inside $B$ is independent of $t$ %for $a\leqslant t
%See \cite{aO73}.
%{\sc A. M. Ostrowski},
%{\em Solution of equations in Euclidean and Banach spaces},
%New York: Academic Press,

\begin{pf*}{Proof of Lemma~\ref{lemmain}}
We will prove a stronger version of the lemma. Namely, we will show
that the statement holds for the family of polynomials
\[
g_t(z)=a_kz^k+a_lz^l+t
\sum_{j\ne k,l}a_jz^j,\qquad 0\leq
t\leq1.
\]
Note that in particular, $g_0(z)=a_kz^k+a_lz^l$ and $g_1(z)=g(z)$.
Let $z\in\C$ be such that $|z|=e^{r-\delta}$. It follows from
(\ref{eqmodulusxikl}) that
\[
t\biggl\llvert\sum_{j\neq k, l}
a_j z^j\biggr\rrvert\leq\sum
_{j\neq k, l} e^{s-jr-h} e^{j(r-\delta)} < %\frac{e^{s-h}}{1-e^{-
n
e^{s-h}.
\]
On the other hand, again by (\ref{eqmodulusxikl}),
\[
\bigl|a_{k}z^{k}+a_{l}z^{l}\bigr| \geq
\bigl|a_{k}z^{k}\bigr|-\bigl|a_{l}z^{l}\bigr| =
e^{s}e^{-\delta k}\bigl(1-e^{-\delta(l-k)}\bigr) > e^{s-\delta n}
\bigl(1-e^{-\delta}\bigr).
\]
Since $ne^{\delta n-h}<1-e^{-\delta}$ holds, everywhere on the circle
$|z|=e^{r-\delta}$ we have
%
%e17 #&#
%
\begin{equation}
\label{eqcondrouche} \bigl|a_{k}z^{k}+a_{l}z^{l}\bigr|
> t\biggl\llvert\sum_{j\neq k, l} a_j
z^j\biggr\rrvert.
\end{equation}
Hence, by Rouch\'e's theorem, the polynomial $g_t$ has exactly $k$
roots in the circle $|z|\leq e^{r-\delta}$.

Let now $z\in\C$ be such that $|z|=e^{r+\delta}$. Then
%
%e18 #&#
%
\begin{equation}
\label{eqestrouche1} t\biggl\llvert\sum_{j\neq k, l}
a_j z^j\biggr\rrvert\leq\sum
_{j\neq k,l} e^{s-jr-h} e^{j(r+\delta)} < %\frac{e^{s-h+\delta
%n}}{1-e^{-\delta}}.
ne^{s-h+\delta n}.
\end{equation}
On the other hand,
\[
\bigl|a_{k}z^{k}+a_{l}z^{l}\bigr| \geq
\bigl|a_{l}z^{l}\bigr|-\bigl|a_{k}z^{k}\bigr| =
e^{s}e^{\delta l}\bigl(1-e^{\delta(k-l)}\bigr) > e^{s}
\bigl(1-e^{-\delta}\bigr).
\]
Therefore, inequality (\ref{eqcondrouche}) also holds everywhere on
the circle $|z|=e^{r+\delta}$.
%$$
%|a_{k}z^{k}+a_{l}z^{l}|
%>\left|\sum_{j\neq k, l} \xi_j z^j\right|.
%$$
It follows from Rouch\'e's theorem that the polynomial $g_t$ has
exactly $l$ roots in the circle $|z|\leq e^{r+\delta}$.
Hence, the polynomial $g_t$ has exactly $l-k$ roots in the ring
$e^{r-\delta}\leq|z|\leq e^{r+\delta}$.\vadjust{\goodbreak}

Let us now show that these $l-k$ roots are located approximately at the
same positions as the nonzero roots of the equation
$a_lz^{l}+a_kz^k=0$. Let $z_0$ be some root of $g_t$ satisfying
$e^{r-\delta}\leq|z_0|\leq e^{r+\delta}$. Then, repeating the
argument of (\ref{eqestrouche1}), we obtain that
%
%e19 #&#
%
\begin{equation}
\label{eqestarg1} \bigl\llvert a_lz_0^{l}+a_kz_0^k
\bigr\rrvert= \biggl\llvert\sum_{j\neq k,l}a_jz_0^j
\biggr\rrvert< n e^{s-h+\delta n}. %\leq
%n e^{(l-k)r} e^{2\delta n-h}.
\end{equation}
Recall that $\varphi=\arg(-a_k/a_l)$. The arguments of the nonzero
roots of the equation $a_lz^{l}+a_kz^k=0$ are given by $\frac{\varphi
+ 2\pi m}{l-k}$, where $m=1,\ldots,l-k$, and their moduli are equal to
$e^{r}$. Let
\[
\varsigma=\min_{m=1,\ldots,l-k}\biggl\llvert\arg z_0-
\frac{\varphi+ 2\pi
m}{l-k}\biggr\rrvert.
\]
Note that $\varsigma\in[0,\frac{\pi}{l-k}]$ by definition. Then
\[
\bigl|\arg\bigl(a_lz_0^{l}\bigr)-\arg
\bigl(-a_kz_0^k\bigr)\bigr|=\bigl|\arg
z_0^{l-k}-\varphi\bigr|=(l-k) \varsigma.
\]
By the inequality $|z_1-z_2|\geq2|z_1|\sin(|{\arg z_1}-\arg z_2|/2)$
valid for $|z_1|\leq|z_2|$ and the inequality $\sin x\geq\frac2 {\pi
} x$ valid for $x\in[0,\frac{\pi} 2]$, we obtain
%
%e20 #&#
%
\begin{equation}
\label{eqestarg2} \bigl\llvert a_lz_0^{l}+
a_kz_0^k\bigr\rrvert\geq2e^{s-\delta l}
\sin\biggl(\frac{(l-k)\varsigma}{2} \biggr) \geq\frac1 {2}e^{s-\delta
n} \varsigma.
\end{equation}
It follows from (\ref{eqestarg1}) and (\ref{eqestarg2}) that
$\varsigma< 2ne^{2\delta n-h}$ and hence $\varsigma<\zeta$.
Therefore, every root $z_0$ of $g_t$ such that $e^{r-\delta}\leq
|z_0|\leq e^{r+\delta}$ is contained in a set of the form (\ref
{eqsetsectorring}) for some $m=1,\ldots,l-k$. To complete the proof,
it remains to show that every set (\ref{eqsetsectorring}) contains
exactly one root of $g_t$. Since $\zeta<\frac{\pi}{l-k}$, all these
sets are disjoint. By the above, $g_t$ does not vanish on their
boundaries. It follows from this and the argument principle that the
number of roots of $g_t$ in any set (\ref{eqsetsectorring}) is
continuous as a function of $t\in[0,1]$ and hence, constant.
Obviously, every set (\ref{eqsetsectorring}) contains exactly one
root of $g_0$ and hence, exactly one root of $g_t$.
\end{pf*}

%s3 #&#
\section{Least concave majorants and weak convergence}\label
{secmajorantconv}

Proposition~\ref{propresnick}\break states the convergence of the point
process $\rho_n$ formed by the logarithms of the coefficients of the
random polynomial $G_n$ to the limiting Poisson process $\rho$. We
will need to deduce from this the weak convergence of certain
functionals of $\rho_n$ to the same functionals of $\rho$. This will
be done using the following well-known continuous mapping theorem;
see \cite{resnickbook}, page 152, or \cite{billingsleybook},
page 30.
%
%pr3.1 #&#
%
\begin{proposition}\label{propcontmapping}
Let $F\dvtx M_1\to M_2$ be a map between two metric spaces $(M_1,d_1)$ and
$(M_2,d_2)$. Let $X_n$ be a sequence of $M_1$-valued random variables
converging weakly to some $M_1$-valued random variable $X$. If $F$ satisfies
\[
\P[ F \mbox{ is discontinuous at } X]=0,
\]
then $F(X_n)$ converges weakly to $F(X)$ on $(M_2,d_2)$.\vadjust{\goodbreak}
\end{proposition}
In order to apply Proposition~\ref{propcontmapping} we need to prove
the a.s. continuity of the functionals under consideration. This is
the aim of the present section. First we introduce some notation.
Let $\MMM$ be the set of locally finite counting measures $\mu$ on
$[0,1]\times(0,\infty]$ such that $\mu([0,1]\times\{\infty\})=0$.
We endow $\MMM$ with the topology of vague convergence. Every $\mu\in
\MMM$ can be written in the form $\mu=\sum_{i}\delta(u_i, v_i)$,
where $i$ ranges in some at most countable index set and $u_i\in
[0,1]$, $v_i\in(0,\infty)$. The number of atoms of $\mu$ in a set of
the form $[0,1]\times[\eps,\infty)$ is finite for every $\eps>0$,
but the atoms of $\mu$ may (and often will) have accumulation points
in the set $[0,1]\times\{0\}$. %An explicit description of this
%topology will be given below.
%This choice of topology prevents the escape of mass to $+\infty$.

The least concave majorant of $\mu\in\MMM$ is a function $\CCC_{\mu
}\dvtx[0,1]\to[0,\infty)$ defined by $\CCC_{\mu}(t)=\inf_f f(t)$,
where the infimum is taken over all concave functions $f\dvtx[0,1]\to
[0,\infty)$ such that $f(u_i)\geq v_i$ for all $i$. We write the
piecewise linear function $\CCC_{\mu}$ in the form
%whose graph passes through atoms of $\mu$ denoted by $(x_k,y_k)$ from
%left to right. The index $k$ ranges in the set $p'<k<p''$, where $p'$
%may be $-\infty$ and $p''$ may be $+\infty$. If $p'$ is finite, then
%the majorant starts at a point $(x_0, )$ We write
%
%e21 #&#
%
\begin{equation}
\label{eqdefLCMmu} \CCC_{\mu}(t)=s_k-r_kt,\qquad t
\in[x_{k}, x_{k+1}],
\end{equation}
where $k$ ranges over a finite or infinite discrete subinterval of $\Z$.
%where we make the convention that $k=-\infty$ is excluded from
%consideration if $p'=-\infty$.
We set $y_k=\CCC_{\mu}(x_k)$. The intervals $[x_k,x_{k+1}]$ (called
the linearity intervals of the majorant) are always supposed to be
chosen in such a way that the points $(x_k,y_k)$ and
$(x_{k+1},y_{k+1})$ are atoms of $\mu$, and there are no further atoms
of $\mu$ on the segment joining these two points.
Fix some small $\kappa\in(0,1/2)$. Given a counting measure $\mu\in
\MMM$, we define the indices $q'=q'_{\kappa}(\mu)$ and
$q''=q''_{\kappa}(\mu)$ by the conditions $x_{q'}\leq\kappa<
x_{q'+1}$ and $x_{q''-1}<1-\kappa\leq x_{q''}$.
%Write $\bar\mu=\mu+\delta(0,0)+\delta(1,0)$ and denote by $\supp\bar
%We will need an analogue of Lemma~\ref{lemPsicont} for majorants
%with infinitely many linearity intervals.

Let $\MMM_1$ be the set of all counting measures $\mu\in\MMM$ with
the following properties:
\begin{longlist}[(3)]
%(0,\infty))=0$;
%
\item[(1)] both $0$ and $1$ are accumulation points for the linearity
intervals of $\CCC_{\mu}$;
\item[(2)] $\mu(L)\leq2$ for every line $L\subset\R^2$;
\item[(3)] no atom of $\mu$ has first coordinate $\kappa$ or
$1-\kappa$.
\end{longlist}
Note that every $\mu\in\MMM_1$ has only simple atoms. Denote by
$\NNN$ the space of finite measures on $\R$ endowed with the weak
topology. Let $V_k(\mu)$\vadjust{\goodbreak} be the subset of $[0,1]\times[0,\infty)$
consisting of $[0,1]\times\{0\}$ together with all atoms of~$\mu$,
except for $(x_k,y_k)$ and $(x_{k+1}, y_{k+1})$.
%
%le3.2 #&#
%
\begin{lemma}\label{lemPsietacont}
The following mappings are continuous on $\MMM_1$:
\begin{longlist}[(3)]
\item[(1)] $\Psi_{1}\dvtx\MMM\to\NNN$ defined by $\Psi_{1}(\mu)=\sum
_{k=q'}^{q''-1} (x_{k+1}-x_k) \delta(r_k)$;
\item[(2)] $H_1\dvtx\MMM\to\R$ defined by $H_1(\mu)=\min\{s_k-r_ku-v\}
$, where the minimum is over $q'\leq k< q''$ and $(u,v)\in V_k(\mu)$;
\item[(3)] $L_1\dvtx\MMM\to\R$ defined by $L_1(\mu)=\min_{q'\leq
k<q''} (x_{k+1}-x_k)$.
\end{longlist}
%
%The mapping $\Psi_{\kappa}:\MMM\to\NNN$ defined by
%$
%$
%is continuous on $\MMM_1$.
\end{lemma}
\begin{pf}
Let $\{\mu_n\}_{n\in\N} \subset\MMM$ be a sequence converging to
$\mu\in\MMM_1$ in the vague topology on $[0,1]\times(0,\infty]$.
Let $\eps>0$ be such that
$
2\eps<\min_{q'\leq k<q''}\{s_{k},\break s_{k}-r_{k}\}
$.
Note that the minimum is strictly positive by the definition of $\MMM
_1$. Denote by $(u_l,v_l)$, where $1\leq l\leq m$, all atoms of $\mu$
(excluding those which are vertices of~$\CCC_{\mu}$) with the
property that $v_l>\eps$.
Since $\mu_n\to\mu$ vaguely, we can find (see
\cite{resnickbook}, Proposition 3.13) atoms of $\mu_n$ denoted by
$(x_{kn}, y_{kn})$ (where $q'\leq k\leq q''$) and $(u_{ln},v_{ln})$
(where $1\leq l\leq m$) such that
%
%e22 #&#
%e23 #&#
%
\begin{eqnarray}
\label{eqlemcont1a} \lim_{n\to\infty} (x_{kn},y_{kn})
&=& (x_k,y_k),\qquad q'\leq k\leq
q'',
\\
\label{eqlemcont1b} \lim_{n\to\infty} (u_{ln},v_{ln})
&=& (u_l,v_l),\qquad 1\leq l\leq m.
\end{eqnarray}
Moreover, since the vague convergence was required to hold on
$[0,1]\times(0,\infty]$, there are no other atoms of $\mu_n$ having
a second coordinate exceeding $2\eps$, provided that $n$ is
sufficiently large. It follows that as $n\to\infty$ and for all
$q'\leq k<q''$,
%
%e24 #&#
%
\begin{equation}
\label{eqlemcont2} r_{kn}:=-\frac{y_{(k+1)n}-y_{kn}}{x_{(k+1)n}-x_{kn}}
\to r_k,\qquad
s_{kn}:=y_{kn}+r_{kn} x_{kn}\to
s_k.
\end{equation}
In particular, for sufficiently large $n$, all $q'\leq k<q''$ and all
$1\leq l\leq m$,
\[
s_{kn}-r_{kn} u_{ln}>v_{ln},\qquad
\inf_{u\in[0,1]} (s_{kn}-r_{kn}u)>2\eps.
\]
It follows that for sufficiently large $n$, the segment joining the
points $(x_{kn}, y_{kn})$ and $(x_{(k+1)n}, y_{(k+1)n})$ belongs to the
majorant of $\mu_n$ for every
$q'\leq k<q''$. Also, $x_{q'n}<\kappa<x_{(q'+1)n}$ and
$x_{(q''-1)n}<1-\kappa< x_{q''n}$. %Moreover, there are no further
%segments of the majorant of $\mu_n$ in the strip $[\kappa,1-\kappa]

By using (\ref{eqlemcont1a}), (\ref{eqlemcont1b}), (\ref
{eqlemcont2}) and letting $\eps\downarrow0$, we obtain that
$H_1(\mu_n)\to H_1(\mu)$ and $L_1(\mu_n)\to L_1(\mu)$ as $n\to
\infty$. This proves the continuity of $H_1$ and $L_1$ on $\MMM_1$.
To prove the continuity of $\Psi_{1}$, note that for every continuous,
bounded function $f\dvtx\R\to\R$,
\[
\int_{\R}f\,d\Psi_{1}(\mu_n) = \sum
_{k=q'}^{q''-1} (x_{(k+1)n}-x_{kn})
f(r_{kn}) \ton\sum_{k=q'}^{q''-1}
(x_{k+1}-x_{k}) f(r_{k}). %=
\]
Thus, $\Psi_{1}(\mu_n) \to\Psi_{1}(\mu)$ weakly, which proves the
continuity of $\Psi_{1}$.\vadjust{\goodbreak}
%We have
%This proves \eqref{eqlemneed1}.
\end{pf}

The next lemma will be needed to prove our main results for $\alpha\in
(0,1)$. Let $\MMM_0$ be the set of all nonzero counting measures $\mu
\in\MMM$ with the following properties:
\begin{longlist}[(3)]
\item[(1)] the number of linearity intervals of $\CCC_{\mu}$ is
finite and $\CCC_{\mu}(0)=\CCC_{\mu}(1)=0$;
\item[(2)] $\bar\mu(L)\leq2$ for every line $L\subset\R^2$, where
$\bar\mu=\mu+\delta(0,0)+\delta(1,0)$;
\item[(3)] no atom of $\mu$ has first coordinate $\kappa$ or
$1-\kappa$.
\end{longlist}
%
%le3.3 #&#
%
\begin{lemma}\label{lemPsicont} The following mappings are
continuous on $\MMM_0$:
\begin{longlist}[(3)]
\item[(1)] $\Psi_0\dvtx\MMM\to\NNN$ defined by
$\Psi_0(\mu)=\sum_{k} (x_{k+1}-x_k) \delta(r_k)$,
where the sum is over all linearity intervals $[x_k,x_{k+1}]$ of the
majorant $\CCC_{\mu}$;
\item[(2)] $H_0\dvtx\MMM\to[0,\infty]$ defined by $H_0(\mu)= \min\{
s_k-r_ku-v\}$, where the minimum is over $q'< k<q''-1$ and $(u,v)\in
V_k(\mu)$;
%$
%H_0(\mu)= \min\{s_k-r_ku-z\dvtx q'< k<q''-1,(u,z)\in(\supp\mu)\bsl V_k
%$ where $V_k$ is the two-element set $\{(x_k,y_k), (x_{k+1}, y_{k+1})
%
\item[(3)] $L_0\dvtx\MMM\to[0,\infty]$ defined by $L_0(\mu)=\min_{q'<
k<q''-1} (x_{k+1}-x_k)$.
\end{longlist}
%
%The mapping $\Psi:\MMM\to\NNN$ defined by
%$
%$
%where the sum is over all linearity intervals $[x_k,x_{k+1}]$ of the
%majorant %$\CCC_{\mu}$,
%is continuous on $\MMM_0$.
\end{lemma}
%
%re3.4 #&#
%
\begin{remark}
In fact, $\Psi_0$ is continuous on the whole of $\MMM$, but we will
not need this. The minimum over an empty set is $+\infty$.
\end{remark}
%
%If $\mu([0,1]\times\{\infty\})\neq0$, then $\mu$ has no well-defined
%majorant. In this case we agree to define $\Psi(\mu)=0$. This
%convention will be used below in the definitions of many other
%mappings without mentioning it explicitly.
%
\begin{pf*}{Proof of Lemma~\ref{lemPsicont}}
Let $\{\mu_n\}_{n\in\N} \subset\MMM$ be a sequence converging
vaguely to $\mu\in\MMM_0$.
The majorant $\CCC_{\mu}$ is a piecewise linear function whose graph
is a broken line connecting the points denoted by $(x_k,y_k)$, where
$p'\leq k\leq p''$ and $(x_{p'},y_{p'})=(0,0)$,
$(x_{p''},y_{p''})=(1,0)$. For $p'<k<p''$, the point $(x_k,y_k)$ is an
atom of $\mu$.
Denote by $(u_l,v_l)$, where $1\leq l\leq m$, all atoms of $\mu$
(excluding those which are vertices of the majorant) with the property
that $v_l>\eps$, where $\eps>0$ is a number such that
$
2\eps<\min_{p'<k<p''-1}\{s_{k}, s_{k}-r_{k}\}
$.
Note that the minimum is taken over the set of linearity intervals of
the majorant excluding the first and the last interval. If the majorant
consists of just two segments, then the minimum is $+\infty$.
The vague convergence $\mu_n\to\mu$ implies (see
\cite{resnickbook}, Proposition 3.13) that we can find atoms of $\mu_n$
denoted by $(x_{kn}, y_{kn})$ (where $p'<k<p''$) and $(u_{ln},v_{ln})$
(where $1\leq l\leq m$) such that
%
%e25 #&#
%e26 #&#
%
\begin{eqnarray}
\label{eqlemcont10a} \lim_{n\to\infty}
(x_{kn},y_{kn})&=&(x_k,y_k),\qquad
p'<k<p'',
\\
\label{eqlemcont10b} \lim_{n\to\infty}
(u_{ln},v_{ln})&=&(u_l,v_l),\qquad
1\leq l\leq m.
\end{eqnarray}
Moreover, if $n$ is sufficiently large, then there are no other atoms
of $\mu_n$ having a second coordinate exceeding $2\eps$. It follows
that as $n\to\infty$ and for all $p'<k<p''-1$,
%
%e27 #&#
%
\begin{equation}
\label{eqlemcont20} r_{kn}:=-\frac
{y_{(k+1)n}-y_{kn}}{x_{(k+1)n}-x_{kn}} \to r_{k},\qquad
s_{kn}:=y_{kn}+r_{kn} x_{kn}\to
s_{k}.
\end{equation}
Note that by concavity $s_k-r_ku_l>v_l$ for all $p'<k<p''-1$ and $1\leq
l\leq m$. Thus, for sufficiently large $n$,
\[
s_{kn}-r_{kn} u_{ln}>v_{ln},\qquad
\inf_{u\in[0,1]} (s_{kn}-r_{kn}u)>2\eps.\vadjust{\goodbreak}
\]
This means that for sufficiently large $n$ the segment joining the
points $(x_{kn}, y_{kn})$ and $(x_{(k+1)n}, y_{(k+1)n})$ belongs to the
majorant of $\mu_n$ for every
$p'< k<p''-1$. Also, $x_{q'n}<\kappa<x_{(q'+1)n}$ and
$x_{(q''-1)n}<1-\kappa< x_{q''n}$.

From (\ref{eqlemcont10a}), (\ref{eqlemcont10b}), (\ref
{eqlemcont20}) we obtain that $H_0(\mu_n)\to H_0(\mu)$ and
$L_0(\mu_n)\to L_0(\mu)$ as $n\to\infty$. This proves
the continuity of $H_0$ and $L_0$ on $\MMM_0$.
To prove the continuity of $\Psi_0$ we need to show that for every
continuous, bounded function $f\dvtx\R\to[0,\infty)$,
%
%e28 #&#
%
\begin{equation}
\label{eqlemneed10} \lim_{n\to\infty}\int_{\R}f \,d
\Psi_0(\mu_n)= \int_{\R}f\,d
\Psi_0(\mu).
\end{equation}
By (\ref{eqlemcont10a}) and (\ref{eqlemcont20}), we have
%
%e29 #&#
%
\begin{equation}
\label{eqlemcont30}
\quad\lim_{n\to\infty} \sum_{p'<k<p''-1}
(x_{(k+1)n}-x_{kn}) f(r_{kn})= \sum
_{p'<k<p''-1} (x_{k+1}-x_{k})
f(r_{k}).
\end{equation}

However, we have to be more careful about approximating the first and
the last segments of $\CCC_{\mu}$.
Denote by $(x_{kn},y_{kn})$, where $k\leq p'+1$, the vertices of the
majorant of $\mu_n$ (counted from left to right) with the property
$x_{kn}\leq x_{(p'+1)n}$. Note that the number of such vertices is, in
general, arbitrary and may be infinite. Since the first segment of the
majorant of $\mu$ joins $(0,0)$ and $(x_{p'+1},y_{p'+1})$, all points
$(u_{ln},v_{ln})$, where $1\leq l\leq m$, are located below the line
joining $(0,0)$ and $(x_{(p'+1)n},y_{(p'+1)n})$ for large $n$.
Therefore, for large $n$ there are no atoms of $\mu_n$ above the line
joining $(0,2\eps)$ and $(x_{(p'+1)n},y_{(p'+1)n})$. Hence,
\begin{eqnarray*}
r_{p'n}:\!&=&-\frac{y_{(p'+1)n}-y_{p'n}}{x_{(p'+1)n}-x_{p'n}}\\
&\in&\biggl[
-\frac{y_{(p'+1)n}-2\eps}{x_{(p'+1)n}}, -
\frac
{y_{(p'+1)n}}{x_{(p'+1)n}} \biggr],\qquad x_{p'n}<2\eps\frac
{y_{(p'+1)n}}{x_{(p'+1)n}}. %\leq4\eps\frac{y_{p'+1}}{x_{p'+1}}.
% k\leq p'.
\end{eqnarray*}
It follows that $r_{p'n}\to r_{p'}$ as $n\to\infty$. The contribution
of linearity intervals to the left of $x_{p'n}$ can be estimated as
follows: for large $n$,
\[
\sum_{k< p'} (x_{(k+1)n}-x_{kn})
f(r_{kn})\leq x_{p'n}\|f\|_{\infty
}\leq4\eps
\frac{y_{p'+1}}{x_{p'+1}}\|f\|_{\infty}.
\]
Since $\eps>0$ can be made as small as we like, we have
%If there are vertices to the left of $(x_{p'n},y_{p'n})$ (if
%$x_{p'n}>0$), then $x_{}$
%
%e30 #&#
%
\begin{equation}
\label{eqlemcont4a} \lim_{n\to\infty} \sum_{k\leq p'}
(x_{(k+1)n}-x_{kn}) f(r_{kn})=x_{p'+1}f(r_{p'}).
\end{equation}
Similar arguments can be applied to the part of the majorant of $\mu_n$
located to the right of $(x_{(p''-1)n}, y_{(p''-1)n})$: with
straightforward notation,
%
%e31 #&#
%
\begin{equation}
\label{eqlemcont4b} \lim_{n\to\infty} \sum_{k\geq p''-1}
(x_{(k+1)n}-x_{kn}) f(r_{kn})=(1-x_{p''-1})f(r_{p''-1}).
\end{equation}
%
%$$
%r_{kn}:=-\frac{y_{(k+1)n}-y_{kn}}{x_{(k+1)n}-x_{kn}} \to r_{p''-1}
%$$
%uniformly in $k\geq p''-1$.
%It follows that
%%f(r_{kn})=x_{p'+1}f(r_{p'}),\label{eqlemcont4a}\\
%%f(r_{kn})=(1-x_{p''-1})f(r_{p''-1}).\label{eqlemcont4b}
Bringing (\ref{eqlemcont30}), (\ref{eqlemcont4a}), (\ref
{eqlemcont4b}) together we obtain (\ref{eqlemneed10}).\vadjust{\goodbreak}
\end{pf*}

In our proofs we will often consider some ``good'' random event
$E_n(\kappa)$ under which we will be able to localize the roots of
$G_n$. The next lemma will be useful.
%
%le3.5 #&#
%
\begin{lemma}\label{lemgoodeventconvdistr}
Let $\{S_n\}_{n\in\N}$ and $S$ be random variables defined on a
common probability space. Suppose that for each $\kappa>0$ we have
random events $\{E_n(\kappa)\}_{n\in\N}$ and random variables $\{
S_n(\kappa)\}_{n\in\N}$, $S(\kappa)$ such that the following
conditions hold:
\begin{longlist}[(4)]
\item[(1)] for every $\kappa>0$, $S_n(\kappa)\to S(\kappa)$ in
distribution as $n\to\infty$;
\item[(2)] $S(\kappa)\to S$ in distribution as $\kappa\downarrow0$;\vadjust{\goodbreak}
\item[(3)] $\lim_{\kappa\downarrow0}\liminf_{n\to\infty}\P
[E_n(\kappa)]=1$;
\item[(4)] $|S_n(\kappa)-S_n|<m_n(\kappa)$ on $E_n(\kappa)$, where
$\lim_{\kappa\downarrow0}\limsup_{n\to\infty} m_n(\kappa)=0$.
\end{longlist}
Then, $S_n \to S$ in distribution as $n\to\infty$.
\end{lemma}
\begin{pf}
Let $f\dvtx\R\to\R$ be a continuous function with compact support. Write
$C=\|f\|_{\infty}$. Take some $\eps>0$. We can choose $\kappa=\kappa
(\eps)>0$ such that
%
%e32 #&#
%
\begin{eqnarray}
\label{eqlemma12} \bigl|\E f\bigl(S(\kappa)\bigr)- \E f(S)\bigr| &<&\eps,
\qquad
\limsup_{n\to\infty}\P\bigl[E_n^c(\kappa)\bigr]<\eps,\nonumber\\[-8pt]\\[-8pt]
\limsup_{n\to\infty} m_n(\kappa)&<&\eps.\nonumber
\end{eqnarray}
Here, $E_n^c(\kappa)$ denotes the complement of $E_n(\kappa)$. After
having fixed $\kappa$ we choose $n_0=n_0(\eps)$ such that for all $n>n_0$,
%
%e33 #&#
%
\begin{equation}
\label{eqlemma14}\quad
\bigl|\E f\bigl(S_n(\kappa)\bigr)- \E f\bigl(S(\kappa)
\bigr)\bigr|< \eps,\qquad \P\bigl[E_n^c(\kappa)\bigr]<2\eps,\qquad
m_n(\kappa)<2\eps.
\end{equation}
%
%It follows that
%$
%|\E f(S_n(\kappa))-\E f(S)|<2\eps.
%$
Denoting by $\omega_f(\delta)=\sup_{|z_1-z_2|\leq\delta}
|f(z_1)-f(z_2)|$ the continuity modulus of~$f$, we have
%
%e34 #&#
%
\begin{equation}
\label{eqlemma16} \quad
\bigl|\E f(S_n)-\E f\bigl(S_n(\kappa)\bigr)\bigr|
\leq\omega_f\bigl(m_n(\kappa)\bigr)+2C\P
\bigl[E_n^c(\kappa)\bigr]\leq\omega_f(2
\eps)+4C\eps.
\end{equation}
Taking $\eps\downarrow0$ in (\ref{eqlemma12}), (\ref
{eqlemma14}), (\ref{eqlemma16}), we obtain $\lim_{n\to\infty}\E
f(S_n)=\E f(S)$.
\end{pf}

%As $n\to\infty$, the random variables $H(\rho_n)$ and $L(\rho_n)$
%converge weakly to %$H(\rho)$ and $L(\rho)$. Both limiting variables
%are strictly positive a.s.

%s4 #&#
\section{\texorpdfstring{Proof of Theorem \protect\ref{theocomplex}}{Proof of Theorem 1.1}}
%s4.1 #&#
\subsection{Notation}\label{subsecproofmainnot}
%We start by introducing some notation.
Let $\xi_0,\xi_1,\ldots$ be i.i.d. random variables
satisfying (\ref{eqtail1}). Consider the least concave majorant $\CCC
_n$ of the set $\{(k,{\log}|\xi_k|)\dvtx
k=0,\ldots,n\}$, where we agree to
exclude points with ${\log}|\xi_k|\leq0$ from consideration. By
definition, $\CCC_n(t)=\inf_f f(t)$ for all $t\in[0,n]$, where the
infimum is taken over all concave functions $f\dvtx[0,n]\to[0,\infty)$
satisfying $f(k)\geq{\log}|\xi_k|$ for all $k=0,\ldots,n$. For
simplicity, we will call $\CCC_n$ the majorant of the polynomial $G_n$.
Denote the vertices of $\CCC_n$ (from left to right) by $(k_{in}, \log
_+|\xi_{k_{in}}|)$, where $0\leq i\leq d_n$ and $k_{0n}=0$,
$k_{d_nn}=n$. On the interval $[k_{in},k_{(i+1)n}]$ the majorant is a
linear function which we write in the form
%
%e35 #&#
%
\begin{equation}
\label{eqdefLCMn} \CCC_n(t)=S_{in}-R_{in}t,\qquad
t\in[k_{in},k_{(i+1)n}],\qquad 0\leq i<d_n.
\end{equation}

Further, denote by $\rho$ a Poisson point process on $[0,1]\times
(0,\infty)$ with intensity $\alpha v^{-(\alpha+1)}\,du\,dv$. The majorant
of $\rho$ is denoted by $\CCC_{\rho}$. As in Section \ref
{subseccomplexroots}, we denote the vertices of $\CCC_{\rho}$,
counted from left to right, by $(X_k,Y_k)$. In the case $\alpha\geq1$
the index $k$ ranges (with probability $1$) in $\Z$ by
Proposition~\ref{propnumberverticesmajorant}.
In the case $\alpha\in(0,1)$ the index $k$ ranges in $p'\leq k\leq
p''$, where $p',p''$ are a.s. finite random variables and
$(X_{p'},Y_{p'})=(0,0)$, $(X_{p''},Y_{p''})=(1,0)$.
On each interval $[X_k,X_{k+1}]$ the majorant $\CCC_{\rho}$ is a
linear function written in the form
%
%e36 #&#
%
\begin{equation}
\label{eqdefLCMPoi} \CCC_{\rho}(t)=S_{k}-R_{k}t,\qquad t
\in[X_{k},X_{k+1}].
\end{equation}

We will be mostly interested in the ``main'' parts of the majorants
$\CCC_n$ and $\CCC_{\rho}$. To make this precise, we take some small
$\kappa\in(0,1/2)$ and let\vadjust{\goodbreak} $0\leq q_n'< q_n''\leq d_n$ and $q'< q''$
be indices (depending on $\kappa$) defined by the conditions
%&
%(1-\kappa)n &\in[k_{(q_n''-1)n}, k_{q_n''n}),\label{eqqn}\\
%&
%1-\kappa&\in[X_{q''-1},X_{q''}).\label{eqq}
%
%e37 #&#
%e38 #&#
%
\begin{eqnarray}
\label{eqqn} k_{q_n'n} &\leq& \kappa n<k_{(q_n'+1)n},\qquad
k_{(q_n''-1)n}<(1-\kappa)n \leq k_{q_n''n},
\\[-2pt]
\label{eqq} X_{q'} &\leq& \kappa<X_{q'+1},\qquad
X_{q''-1}<1-\kappa\leq X_{q''}.
\end{eqnarray}

In our proof of Theorem~\ref{theocomplex} it will be convenient to
consider the logarithms of the roots of $G_n$ rather than the roots
themselves. We will prove the following weak convergence of random
probability measures on the space $E=[-\infty,\infty]\times[0,2\pi]$:
%
%e39 #&#
%
\begin{equation}
\label{eqmainrestated} \frac1n\sum_{z\in\ZZZ_n}
\delta\bigl({b_n \log}|z|, \arg z\bigr) \toweak\sum_{k}(X_{k+1}-X_{k})
\lambda_{R_k},
\end{equation}
where $\lambda_r$ is the Lebesgue measure on $\{r\}\times[0,2\pi]$
normalized to have total mass~$1$. The sum on the right-hand side is
over all linearity intervals $[X_k,X_{k+1}]$ of the majorant $\CCC_{\rho
}$. To see that (\ref{eqmainrestated}) implies the statement
of Theorem~\ref{theocomplex} note that the map $F\dvtx E\to\bar\C$
given by $F(r,\varphi)=e^{r+ i \varphi}$ is continuous, and hence it
induces a weakly continuous map between the corresponding spaces of
probability measures; see~\cite{resnickbook}, Proposition 3.18. By
Proposition~\ref{propcontmapping} we can apply $F$ to the both sides
of (\ref{eqmainrestated}) which yields Theorem~\ref{theocomplex}.
So, let $f\dvtx E\to[0,\infty)$ be a continuous function. To prove
Theorem~\ref{theocomplex} it suffices to show that
%
%e40 #&#
%
\begin{equation}
\label{eqproofmain} S_n:=\frac1 n\sum_{z\in\ZZZ_n}
f \bigl({b_n\log}|z|, \arg z \bigr) \todistr\sum_{k}(X_{k+1}-X_{k})
\bar f(R_k) =:S,
\end{equation}
where $\bar f\dvtx[-\infty, \infty]\to\R$ is defined by $\bar f(r)=\int
_E f\,d\lambda_r=\frac1 {2\pi}\int_{0}^{2\pi} f (r,\varphi
)\,d\varphi$.

%The continuity modulus of $f$ is denoted by
%$$
%$$

We will need to consider the cases $\alpha\geq1$ and $\alpha\in
(0,1)$ separately. The main difference is that in the former case the
linearity intervals of the majorant $\CCC_{\rho}$ cluster at $0$ and
$1$, whereas in the latter case we have a well-defined first and a
well-defined last linearity interval of $\CCC_{\rho}$. These
intervals cannot be ignored and have to be considered separately. This
makes the case $\alpha\in(0,1)$ somewhat more difficult.\vspace*{-2pt}

%s4.2 #&#
\subsection{\texorpdfstring{Proof in the case $\alpha\geq1$}{Proof in the case alpha >= 1}}
%Consider the case $\alpha\geq1$ first.
The next lemma shows that with probability approaching $1$ the majorant
of $G_n$ has some ``good'' properties. In particular, there is a gap
between the majorant and the points lying below the majorant. Let
$W_{in}\subset[0,n]\times[0,\infty)$ be the set consisting of
$[0,n]\times\{0\}$ together with the points $(k, \log_+ |\xi_k|)$
for all $0\leq k\leq n$ such that $k\neq k_{in},k_{(i+1)n}$.\vspace*{-2pt}
%
%le4.1 #&#
%
\begin{lemma}\label{lemEngeq1}
Fix sufficiently small $\eps>0$, and consider a random event
$E_n:=E_n^1\cap E_n^2$, where
%$E_n=E_n(\kappa):=E_n^1\cup E_n^1$, where
%
%e41 #&#
%e42 #&#
%
\begin{eqnarray}
%E_n^1&=\left\{\min_{q_n'\leq i<q_n''} \min_{k\neq k_{in},k_{(i+1)n}}
%(S_{in}-R_{in}k-\log_+ |\xi_k|)>n^{\frac1 {\alpha}-\eps}\right\},
\label{eqEn1geq1} E_n^1&=&
\Bigl\{\min_{q_n'\leq i<q_n''} \min_{(u,v)\in W_{in}}
(S_{in}-R_{in}u-v)>n^{1/{\alpha}-\eps}
\Bigr\},
\\[-2pt]
\label{eqEn2geq1} E_n^2&=& \Bigl\{\min_{q'_n\leq i<q''_n}
(k_{(i+1)n}-k_{in})>\sqrt n \Bigr\}.
\end{eqnarray}
Then, $\lim_{n\to\infty}\P[E_n]=1$.\vadjust{\goodbreak}% if $\eps$ is sufficiently
%small.
\end{lemma}
\begin{pf}
By\vspace*{1pt} Proposition~\ref{propresnick} the point process $\rho_n=\sum
_{k=0}^n\delta(\frac kn, \frac{{\log}|\xi_k|}{a_n})$ converges to
$\rho$ weakly on $\MMM$, where the points with ${\log}|\xi_k|\leq0$
are ignored. Recall the definition of the functionals $H_1$ and $L_1$
in Lemma~\ref{lemPsietacont}. By scaling,
\begin{eqnarray*}
H_1(\rho_n)&=&\frac1{a_n}\min_{q_n'\leq i<q_n''}
\min_{(u,v)\in
W_{in}} (S_{in}-R_{in}u-v),
\\
L_1(\rho_n)&=&\frac1n \min_{q'_n\leq i<q''_n}
(k_{(i+1)n}-k_{in}).
\end{eqnarray*}
%
%H_1(\rho_n)&=\frac1{a_n}\min_{q_n'\leq i<q_n''} \min_{(u,z)\in W_i}
%%(S_{in}-R_{in}u-z)\todistr H_1(\rho),\\
%L_1(\rho_n)&=\frac1n \min_{q'_n\leq i<q''_n} (k_{(i+1)n}-k_{in})
It follows that
\[
\P\bigl[E_n^1\bigr]=\P\bigl[H_1(
\rho_n)>a_n^{-1}n^{1/{\alpha}-\eps}\bigr],\qquad
\P\bigl[E_n^2\bigr]=\P\bigl[L_1(
\rho_n)>n^{-1/2}\bigr].
\]
By Lemma~\ref{lemPsietacont} and Proposition
\ref{propcontmapping} (which is applicable since $\P[\rho\in\MMM_1]=1$
for $\alpha\geq1$), we have $H_1(\rho_n)\to H_1(\rho)$ and
$L_1(\rho_n)\to L_1(\rho)$ in distribution as $n\to\infty$.
Note that $H_1(\rho)>0$ and $L_1(\rho)>0$ a.s. Also, $a_n>n^{1/
{\alpha}-{\eps}/{2}}$ for large $n$ by (\ref{eqtail1})
and~(\ref{eqdefan}). It follows that $\lim_{n\to\infty}\P[E_n]=1$.
\end{pf}
In the next lemma we localize most complex roots of $G_n$ under the
event~$E_n$.

%le4.2 #&#
%
\begin{lemma}\label{lemlocalzerogeq1}
On the random event $E_n$ the following holds: for every $q_n'\leq
i<q_n''$ and $1\leq m\leq k_{(i+1)n}-k_{in}$ there is exactly one root
of $G_n$ in the set
\[
Z_{i,m}(n):= \biggl\{z\in\C\dvtx\bigl|{\log}|z|-R_{in}\bigr|<
\delta_n, \biggl\llvert\arg z-\frac{\varphi_{in}+ 2\pi
m}{k_{(i+1)n}-k_{in}}\biggr\rrvert<
\delta_n \biggr\},
\]
where $\delta_n=\exp(-n^{1/{\alpha}-2\eps})$ and $\varphi_{in}=\arg
(-\xi_{k_{in}}/\xi_{k_{(i+1)n}})$. The above sets are
disjoint, and there are no other roots in the ring $R_{q_n'n}-\delta
_n\leq{\log}|z|< R_{(q_n''-1)n}+\delta_n$.
%These roots are different and there exist no other roots of $G_n$ in
%the ring $R_{q_n'n}-\delta_n\leq{\log}|z|\leq R_{(q_n''-1)n}+\delta_n$.
%R_{q_n'n}-\delta_n\leq{\log}|z|\leq R_{(q_n''-1)n}+\delta_n.
%contains are no other complex roots of $G_n$. %in the ring $b_n|{\log}
%|z|- R_{in}|\leq1/n^2$.
\end{lemma}
\begin{pf}
First note that on $E_n$ it is impossible that $q_n'=0$ and ${\log}|\xi
_0|\leq0$. Similarly, on $E_n$ it is impossible that $q_n''=d_n$ and
${\log}|\xi_n|\leq0$. It follows from (\ref{eqEn1geq1}) that on the
event $E_n$ the conditions of Lemma~\ref{lemmain} are fulfilled for
the polynomial $G_n$ with $k=k_{in}$, $l=k_{(i+1)n}$, $\delta=\zeta
=\delta_n$ for every $q_n'\leq i<q_n''$. Hence, every set $Z_{i,m}(n)$
contains exactly one root of $G_n$. Also, it follows from the proof of
Lemma~\ref{lemmain} that there are exactly $k_{q_n'n}$ roots of $G_n$
in the disk ${\log}|z|<R_{q_n'n}-\delta_n$ and exactly $k_{q_n''n}$
roots in the disc ${\log}|z|<R_{(q_n''-1)n}+\delta_n$. Hence, there are
exactly $k_{q_n''n}-k_{q_n'n}$ roots in the ring $R_{q_n'n}-\delta_n\leq
{\log}|z|< R_{(q_n''-1)n}+\delta_n$, which coincides with the
number of different sets $Z_{i,m}(n)$. It remains to show that the sets
$Z_{i,m}(n)$ are disjoint on $E_n$. To this end, it suffices to show
that on $E_n$ it holds that $R_{(i+1)n}-R_{in}>3\delta_n$ for every
$q_n'\leq i<q_n''-1$. We have
\[
(k_{(i+2)n}-k_{(i+1)n}) (R_{(i+1)n}-R_{in})=
S_{in}-R_{in}k_{(i+2)n}-{\log}|\xi_{k_{(i+2)n}}|>n^{1/{\alpha
}-\eps}
\]
on $E_n$. Since $k_{(i+2)n}-k_{(i+1)n}\leq n$, this implies that which
is required.\vadjust{\goodbreak}
\end{pf}
Our aim is to show that $S_n\to S$ in distribution as $n\to\infty$;
see (\ref{eqproofmain}). Define random variables $S_n(\kappa)$ and
$S(\kappa)$ which approximate $S_n$ and $S$ by
\begin{eqnarray*}
S_n(\kappa)&=&\frac1n \sum_{q_n'\leq i<q_n''}(k_{(i+1)n}-k_{in})
\bar f (b_nR_{in} ),
\\
S(\kappa)&=&\sum_{q'\leq i<q''} (X_{i+1}-X_i)
\bar f(R_{i}).
\end{eqnarray*}
Let\vspace*{1pt} $\omega_f(\delta)={\sup_{|z_1-z_2|\leq\delta}}|f(z_1)-f(z_2)|$,
where $\delta>0$, be the continuity modulus of the function $f$.
%
%le4.3 #&#
%
\begin{lemma}
On the random event $E_n$ it holds that
\[
\bigl|S_n-S_n(\kappa)\bigr|\leq\omega_f\bigl( 10/\sqrt
n\bigr)+2\kappa\|f\|_{\infty}.
\]
%
%$$
%S_n-S_n(\kappa)\leq
%$$
\end{lemma}
\begin{pf}
We always assume that the event $E_n$ occurs. Take some $q'_n\leq
i<q''_n$. By Lemma~\ref{lemlocalzerogeq1}, the polynomial $G_n$ has
a unique root, denoted by $z_{i,m}(n)$, in the set $Z_{i,m}(n)$, where
$1\leq m\leq\Delta_{in}$ and $\Delta_{in}=k_{(i+1)n}-k_{in}$.
Denote by $\ZZZ_{in}$ the finite set $\{z_{i,m}(n)\dvtx1\leq m\leq
\Delta_{in}\}$. By (\ref{eqEn2geq1}) we have $\Delta_{in}>\sqrt n$. By
the definition of $Z_{i,m}(n)$ in Lemma~\ref{lemlocalzerogeq1},
\[
\biggl\llvert f\bigl({b_n\log}\bigl|z_{i,m}(n)\bigr|,\arg
z_{i,m}(n)\bigr)- \frac{\Delta_{in}}{2\pi}\int_{({\varphi_{in}+2\pi
m-\pi})/{\Delta
_{in}}}^{({\varphi_{in}+2\pi m+\pi})/{\Delta
_{in}}}f(b_nR_{in},
\varphi)\,d\varphi\biggr\rrvert
\]
is smaller than $\omega_f(10/\sqrt n)$.
Taking the sum over $1\leq m\leq\Delta_{in}$, we obtain
%
%e43 #&#
%
\begin{equation}
\label{eqwspomest1geq1} \frac1n\biggl\llvert\sum_{z\in\ZZZ_{in}}
f \bigl({b_n\log}|z|, \arg z \bigr)- \Delta_{in}\bar f
(b_nR_{in} )\biggr\rrvert\leq\frac{\Delta_{in}}n
\omega_f\bigl(10/{\sqrt n}\bigr).
\end{equation}
Let $\ZZZ_n^*$ be the set of roots (counted with multiplicities) of
the polynomial $G_n$ not belonging to $\bigcup_{q_n'\leq i< q_n''} \ZZZ
_{in}$. The number\vspace*{1pt} of roots in $\ZZZ_{n}^*$ is
$n-k_{q_n''n}+k_{q_n'n}$, which is at most $2\kappa n$ by
(\ref{eqqn}). Hence,
%
%e44 #&#
%
\begin{equation}
\label{eqwspomest2geq1} %\left|
\frac1 n\sum_{z\in\ZZZ_n^*}
f \bigl({b_n\log}|z|, \arg z \bigr) %\right|
\leq2\kappa\|f\|_{\infty}.
\end{equation}
%
%Also, by \eqref{eqqn} we have %and \eqref{eqq} we have
Taking the sum of (\ref{eqwspomest1geq1}) over all $q'_n\leq
i<q''_n$ and applying (\ref{eqwspomest2geq1}), we obtain the
required inequality.
\end{pf}
%
%le4.4 #&#
%
\begin{lemma}\label{lemconvmajorpointprocgeq1}
We have $S_n(\kappa)\to S(\kappa)$ in distribution as $n\to\infty$.
%The following distributional convergence holds: %$S_n(\kappa)\todistr
%S(\kappa)$, where
%$$
%S_n(\kappa):= \sum_{q_n'\leq i<q_n''}\frac{k_{(i+1)n}-k_{in}}{n} \bar
%%f\left(b_nR_{in}\right)
%=:S(\kappa).
%$$
%$$
%frac 1n \sum_{q_n'(\kappa)\leq i<q_n''(\kappa)}(k_{(i+1)n}-k_{in})
%$$
\end{lemma}
\begin{pf}
By Proposition~\ref{propresnick} the point process $\rho_n=\sum
_{k=0}^n\delta(\frac kn, \frac{{\log}|\xi_k|}{a_n})$ converges to
$\rho$ weakly on $\MMM$. By Lemma~\ref{lemPsietacont} and
Proposition~\ref{propcontmapping} (which is applicable since $\P
[\rho\in\MMM_1]=1$ for $\alpha\geq1$), we obtain that $\Psi_{1}(\rho
_n)$ converges weakly (as a random finite measure on $\R$)
to $\Psi_{1}(\rho)$. This implies that
$
\int_{\R}\bar f \,d \Psi_{1}(\rho_n)
$
converges in distribution to
$
\int_{\R}\bar f \,d \Psi_{1}(\rho),
$
which is exactly what is stated in the lemma.\vadjust{\goodbreak}
\end{pf}

The proof of Theorem~\ref{theocomplex} in the case $\alpha\geq1$
can be completed as follows.
Recall that $\lim_{n\to\infty}\P[E_n]=1$ by Lemma
\ref{lemEngeq1}. Trivially, $S(\kappa)\to S$ as $\kappa\downarrow0$
a.s. and hence, in distribution. By Lemma
\ref{lemgoodeventconvdistr} (whose conditions have been verified above)
we obtain that $S_n\to S$ in distribution as $n\to\infty$. This
proves~(\ref{eqproofmain}).
%Applying Lemma~\ref{lemconvmajorpointprocgeq1} together with
%that \eqref{eqproofmain} holds, as required.

%s4.3 #&#
\subsection{\texorpdfstring{Proof in the case $\alpha\in(0,1)$}{Proof in the case alpha in (0,1)}}
This case is somewhat more difficult since we have to analyze the first
and the last segment of the majorant of $G_n$ separately. In our proof
we will assume that $\xi_0\neq0$ a.s. This assumption will be removed
afterward. Let $0<\tau_n\leq n$, $0\leq\theta_n< n$ be indices (for
concreteness, we choose the smallest possible values) such that
\[
\frac{{\log}|\xi_{\tau_n}|}{\tau_n}=\max_{k=1,\ldots,n} \frac
{{\log}|\xi_k|}{k},\qquad \frac{{\log}|\xi_{\theta_n}|}{n-\theta_n}=
\max_{k=0,\ldots,n-1} \frac{{\log}|\xi_k|}{n-k}.
\]
%
%It is possible that some coefficients of $G_n$ are zero, which causes
%some difficulties especially if the first or last coefficient is zero.
%It will be convenient to introduce the following notation: we let
%$k_{0n}'$ (respectively, $k_{d_nn}'$) to be the index of the first
%(respectively, last) non-zero coefficient of $G_n$. For all $1\leq i
Recall that $W_{in}\subset[0,n]\times[0, +\infty)$ denotes the set
consisting of $[0,n]\times\{0\}$ together with the points $(k, \log_+
|\xi_k|)$ for all $0\leq k\leq n$ such that $k\neq k_{in},k_{(i+1)n}$.
%
%le4.5 #&#
%
\begin{lemma}\label{lemEn01}
For sufficiently small $\eps>0$ and $\kappa\in(0,1/2)$, consider a
random event $E_n:=\bigcap_{i=1}^6 E_n^i$, where
%
%e45 #&#
%e46 #&#
%e47 #&#
%e48 #&#
%e49 #&#
%e50 #&#
%
\begin{eqnarray}
\label{eqEn101} E_n^1&=& \Bigl\{\min_{0<i<d_n-1}
\min_{(u,v)\in W_{in}} (S_{in}-R_{in}u-v)>n^{1/{\alpha}-\eps}
\Bigr\},
\\
\label{eqEn201} E_n^2&=& \Bigl\{\min_{0\leq i<d_n}
(k_{(i+1)n}-k_{in})>\sqrt n \Bigr\},
\\
\label{eqEn301} E_n^3&=& \biggl\{\min_{j\neq0,\tau_n}
\biggl(\frac{{\log}|\xi_{\tau
_n}|}{\tau_n}-\frac{{\log_+}|\xi_j|}{j} \biggr)>n^{1/{\alpha
}-1-\eps} \biggr\},
\\
\label{eqEn401} E_n^4&=& \biggl\{\min_{j\neq n,\theta_n}
\biggl(\frac{{\log}|\xi_{\theta_n}|}{n-\theta_n}-\frac{{\log_+}|\xi
_j|}{n-j} \biggr)>n^{1/{\alpha}-1-\eps} \biggr\},
\\
\label{eqEn501} E_n^5&=&\bigl\{\tau_n>\kappa
n, \theta_n<(1-\kappa)n\bigr\},
\\
\label{eqEn601} E_n^6&=& \bigl\{ %\max({\log}|\xi_0|, {\log}|\xi_n|, {\log}|
\bigl|{\log}|\xi_0|\bigr|<n^{\eps}, \bigl|{\log}|\xi_n|\bigr|<n^{\eps}
\bigr\}.
\end{eqnarray}
Then $\lim_{\kappa\downarrow0}\liminf_{n\to\infty} \P[E_n]=1$
for every $\eps>0$.
\end{lemma}
%
%re4.6 #&#
%
\begin{remark}
Note that $E_n^1$ states that all segments of the majorant, except for
the first and the last one, are well separated from the points below
the majorant. For the first and the last segment the well-separation
property is stated in random events $E_n^3$ and $E_n^4$. %The event
%$E_n^2$ states that the linearity intervals of the majorant are not
%too short.
\end{remark}
%
%re4.7 #&#
%
\begin{remark}\label{remEn01}
We will see that on $E_n^3\cap E_n^6$ the segment joining the points
$(0,{\log_+}|\xi_0|)$ and $(\tau_n, {\log}|\xi_{\tau_n}|)$ is the
first segment of the majorant of $G_n$. In general, this segment need
not be the first one, for example, if ${\log_+}|\xi_0|$ is very large.
Similarly, on $E_n^4\cap E_n^6$ the segment joining $(\theta_n, {\log}
|\xi_{\theta_n}|)$ and $(n, {\log_+}|\xi_n|)$ is the last segment of
the majorant of $G_n$. It follows that $q_n'=0$ and $q_n''=d_n$ on the
event $\bigcap_{i=3}^6 E_n^i$.
\end{remark}
\begin{pf*}{Proof of Lemma~\ref{lemEn01}}
We start by considering $E_n^3$. Let $\tilde\rho$ be a Poisson point
process on $(0,\infty)$ with intensity $\frac{\alpha} {1-\alpha}
v^{-(\alpha+1)}\,dv$. We will show that the following weak convergence
of point processes on $(0,\infty]$ holds:
%
%e51 #&#
%
\begin{equation}
\label{eqwspomconvpoiss} %\tilde\rho_n:=
\tilde\rho_n:=\sum
_{k=1}^n\delta\biggl(\frac{{b_n \log}|\xi_k|}{k} \biggr)
\toweak\tilde\rho.
\end{equation}
Again, we agree that the terms with ${\log}|\xi_k|\leq0$ are ignored.
Recall from (\ref{eqdefan}) that $\bar F(a_n)\sim1/n$ as $n\to
\infty$. Take some $t>0$. By (\ref{eqtail1}) and a well-known
uniform convergence theorem for regularly varying functions we have,
uniformly in $\kappa n\leq k\leq n$,
%
%e52 #&#
%
\begin{equation}
\label{eqwspom101} \P\biggl[\frac{{b_n\log}|\xi_k|}{k}>t \biggr] %\sim
= \bar F
\biggl(\frac{kta_n}{n} \biggr) \sim n^{\alpha-1}k^{-\alpha}t^{-\alpha},
\qquad
n\to\infty.
\end{equation}
To estimate the terms with $1\leq k\leq\kappa n$ recall the following
Potter bound: for every small $\delta>0$ we have $\bar F(x)/\bar
F(y)\leq2(x/y)^{-\alpha-\delta}$ as long as $x<y$ are sufficiently
large; see \cite{binghambook}, Theorem 1.5.6. We have
%
%e53 #&#
%
\begin{eqnarray}
\label{eqwspom201} \sum_{k=1}^{[\kappa n]}\P
\biggl[\frac{{b_n\log}|\xi_k|}{k}>t \biggr] &=& \sum_{k=1}^{[\kappa n]}
\bar F \biggl(\frac{kta_n}{n} \biggr)
\nonumber\\
&\leq& 2\bar F(\kappa ta_n) \sum_{k=1}^{[\kappa n]}
\biggl(\frac{\kappa
n}{k} \biggr)^{\alpha+\delta}
\\
&<& C \kappa^{1-\alpha}t^{-\alpha}.
\nonumber
\end{eqnarray}
From (\ref{eqwspom101}) and (\ref{eqwspom201}) with $\kappa
\downarrow0$, we get
%
%e54 #&#
%
\begin{equation}
\label{eqwspom301} \lim_{n\to\infty}\sum_{k=1}^n
\P\biggl[\frac{{b_n\log}|\xi_k|}{k}>t \biggr]=\frac{1}{(1-\alpha
)}t^{-\alpha}.
\end{equation}
%
%Take $0<t_1<t_2$. It follows that
%$$
%{b_n{\log}%|\xi_k|}{k}<t_2\right]\right)\ton%\exp\left\{-\frac{t_1^{-
%$$
By a standard argument this implies (\ref{eqwspomconvpoiss}). Since
the weak convergence of point processes in (\ref{eqwspomconvpoiss})
implies (via Proposition~\ref{propcontmapping}) the weak convergence
of the corresponding upper order statistics, we have
\[
\min_{j\neq0,\tau_n} \biggl\{b_n \biggl(\frac{{\log}|\xi_{\tau
_n}|}{\tau_n}-
\frac{{\log_+}|\xi_j|}{j} \biggr) \biggr\}\todistr\tilde V_{1}-\tilde
V_2,
\]
where $\tilde V_1,\tilde V_2$ are the largest and the second largest
points of $\tilde\rho$. Since $b_n^{-1}>n^{1/{\alpha}-1-
{\eps}/{2}}$ for large $n$ and since $\tilde V_1>\tilde V_2$ a.s., we
have $\lim_{n\to\infty}\P[E_n^3]=1$. By symmetry, $\lim_{n\to
\infty}\P[E_n^4]=1$.

Let us consider $E_n^5$. By (\ref{eqwspomconvpoiss}) and
(\ref{eqwspom201}) we have, for every $t>0$ and sufficiently large $n$,
\begin{eqnarray*}
\P[\tau_n\leq\kappa n] &\leq& \P\biggl[\max_{k=1,\ldots,n}
\frac{{b_n\log}|\xi_k|}{k}\leq t \biggr]+\P\biggl[\max_{k=1,\ldots
,[\kappa n]}
\frac{{b_n\log}|\xi_k|}{k}>t \biggr]
\\
&<& 2\exp\biggl\{-\frac{1}{1-\alpha}t^{-\alpha} \biggr\}+C
\kappa^{1-\alpha}t^{-\alpha}.
\end{eqnarray*}
Taking $t^{\alpha}=\kappa^{(1-\alpha)/2}$ and letting $\kappa
\downarrow0$, we obtain $\lim_{\kappa\downarrow0}\limsup_{n\to
\infty}\P[\tau_n\leq\kappa n]=0$. By symmetry,\vspace*{1pt} $\lim_{\kappa
\downarrow0}\liminf_{n\to\infty}\P[E_n^5]=1$. Since we assume that
$\xi_0\neq0$ a.s., we have $\lim_{n\to\infty}\P[E_n^6]=1$.

%Before we proceed further we will show that on the event $E_n^3\cup
%E_n^6$ the segment joining the points $(0,\log_+|\xi_0|)$ and $(
%$G_n$. Similarly, on $E_n^4\cup E_n^6$ the segment joining $(\theta_n,
%{\log}|\xi_{\theta_n}|)$ and $(n, \log_+|\xi_n|)$ is the last segment
%of the majorant of $G_n$.

To proceed further we need to prove Remark~\ref{remEn01}. Let
$s,r\in\R$ be such that $s={\log_+}|\xi_0|$ and $s-\tau_n r={\log}
|\xi_{\tau_n}|$. On the random event $E_n^3\cap E_n^6$ we have that
for every $1\leq j\leq n$, $j\neq\tau_n$,
\[
s-jr-{\log}|\xi_j|=j \biggl(\frac{{\log}|\xi_{\tau_n}|}{\tau_n}-\frac{{\log}
|\xi_j|}{j}-s
\biggl(\frac{1}{\tau_n}-\frac1j \biggr) \biggr) > n^{1/{\alpha
}-1-\eps}-2n^{\eps}
>0.
\]
This proves what is required.
%Notice the following corollary: on the event $\bigcup_{i=3}^6 E_n^i$ we
%have $q_n'=0$ and $q_n''=d_n$.
%Consider
%h^*:=\min_{j\neq0,\tau_n} \frac{s-jr-{\log}|\xi_j|}{j}
%=
%{\log}|\xi_j|}{j}-s\left(\frac{1}{\tau_n}-\frac%1j\right)\right).
%On $E_n^3\cup E_n^6$ it holds that $h^*>0$, which proves the statement.

Let us turn our attention to $E_n^1$ and $E_n^2$. By Proposition
\ref{propresnick} the point process $\rho_n=\sum_{k=0}^n \delta(\frac k
n, \frac{{\log}|\xi_k|}{a_n})$ converges weakly to $\rho$. Recall\vspace*{1pt} the
definition of the functionals $H_0$ and $L_0$ in Lemma
\ref{lemPsicont}. By a scaling argument,
\begin{eqnarray*}
H_0(\rho_n)&=&\frac1{a_n}
\min_{q_n'< i<q_n''-1} \min_{(u,v)\in
W_{in}} (S_{in}-R_{in}u-v),
\\
L_0(\rho_n)&=&\frac1n \min_{q'_n< i<q''_n-1}
(k_{(i+1)n}-k_{in}). %=\min_{q'\leq i<q''} \min_{j\neq i,i+1}
%(S_{i}-R_{i}X_j-{\log}|\xi_j|).
\end{eqnarray*}
As observed in Remark~\ref{remEn01}, on the event $\bigcap_{i=3}^6
E_n^i$ we have $q_n'=0$ and $q_n''=d_n$. Hence,
\begin{eqnarray*}
\P\bigl[E_n^1\bigr]&\geq&\P\bigl[H_0(
\rho_n)>a_n^{-1} n^{1/{\alpha}-\eps
}\bigr]-
\Biggl(1-\P\Biggl[\bigcap_{i=3}^6
E_n^i\Biggr]\Biggr),
\\
\P\bigl[E_n^2\bigr]&\geq&\P\bigl[L_0(
\rho_n)>n^{-1/2}\bigr]-\Biggl(1-\P\Biggl[\bigcap
_{i=3}^6 E_n^i\Biggr]
\Biggr).
\end{eqnarray*}
By Lemma~\ref{lemPsicont} and Proposition~\ref{propcontmapping} (which
is applicable since $\P[\rho\in\MMM_0]=1$ for $\alpha\in (0,1)$), we
have $H_0(\rho_n)\to H_0(\rho)$ and $L_0(\rho_n)\to L_0(\rho)$ weakly
on $[0,\infty]$ as $n\to\infty$. Note that $H_0(\rho)>0$ and
$L_0(\rho)>0$ a.s. and $a_n>n^{1/{\alpha}-{\eps}/{2}}$ for
large~$n$. Also, we have already shown that the probability of the
event $\bigcap_{i=3}^6 E_n^i$ can be made arbitrary close to $1$ by
choosing $\kappa$ small and $n$ large. It follows that
$\lim_{n\to\infty}\P[E_n^1]=\lim_{n\to\infty}\P [E_n^2]=1$, as
required.
%Since a.s. we have
%$$
%$$
%To complete the proof of $\lim_{n\to\infty}\P[E_n^2]=1$ note that on
%the event $\bigcup_{i=3}^6 E_n^i$ we have $k_{1n}=\tau_n> \kappa n$ and
%$k_{(n-1)n}=\theta_n< (1-\kappa)n$.
\end{pf*}

In the next lemma we isolate all roots of $G_n$ under the event $E_n$.
It will be convenient to modify the definition of the slopes of the
majorant of $G_n$. Let $R_{0n}'$ be such that ${\log}|\xi
_0|-R_{0n}'k_{1n}={\log}|\xi_{k_{1n}}|$. This is well-defined since
$\xi_0\neq0$ a.s. Note that if ${\log}|\xi_0|<0$, then $R_{0n}'$ is
not the same as $R_{0n}$. On $E_n$ we have the estimate
%
%e55 #&#
%
\begin{equation}
\label{eqdiffRinRinprime} \bigl|R_{0n}-R_{0n}'\bigr|
\leq\tau_n^{-1}\bigl|{\log}|\xi_0|\bigr| %\left|\frac{{\log}|\xi_{\tau_n}|-\log_+|
<n^{2\eps-1}.
\end{equation}
In a similar way, we can define $R_{(d_n-1)n}'$. For all $0<i<d_n-1$,
set $R_{in}'=R_{in}$.

%le4.8 #&#
%
\begin{lemma}\label{lemisolroots01}
On the random event $E_n$ the following holds: for every $0\leq i<d_n$
and $1\leq m\leq k_{(i+1)n}-k_{in}$, there is exactly one root of $G_n$
in the set
\[
Z_{i,m}(n):= \biggl\{z\in\C\dvtx\bigl\llvert{\log}|z|-R_{in}'
\bigr\rrvert< \delta_{n}, \biggl\llvert\arg z-\frac{\varphi_{in}+
2\pi m}{k_{(i+1)n}-k_{in}}
\biggr\rrvert< \delta_{n} \biggr\},
\]
where $\varphi_{in}=\arg(-\xi_{k_{in}}/\xi_{k_{(i+1)n}})$ and
$\delta_{n}=\exp(-n^{{1}/{\alpha}-1-3\eps})$. The above sets
are disjoint, and there are no other roots of $G_n$.
\end{lemma}
\begin{pf}
Consider the case $i=0$ first. Let $s={\log}|\xi_{0}|$ (well defined
since $\xi_0\neq0$ a.s.) and $r=R_{0n}'$. Note that $\tau_n=k_{1n}$
on $E_n$ by Remark~\ref{remEn01}. In order to apply Lemma
\ref{lemmain} with $k=0$, $l=\tau_n$ we need to estimate $h:=\min_{j\neq
0,\tau_n} (s-jr-{\log}|\xi_j|)$. On the event $E_n$ we have
\begin{eqnarray*}
\min_{j\neq0,\tau_n}\frac{s-jr-{\log}|\xi_j|}{j} &=& \min_{j\neq0,\tau_n}
\biggl(
\frac{{\log}|\xi_{\tau_n}|}{\tau_n}-\frac{{\log}|\xi_j|}{j}-s \biggl(\frac
{1}{\tau_n}-\frac1j \biggr)
\biggr)
\\
&>& n^{1/{\alpha}-1-2\eps},
\end{eqnarray*}
which implies that $h>n^{1/{\alpha}-1-2\eps}$.
To prove the lemma for $i=0$, apply Lem\-ma~\ref{lemmain} with $k=0$,
$l=\tau_n$ and $\delta=\zeta=\delta_n$. The case $i=d_n-1$ is similar.
%It follows that the set
%$$
%Z_{0,m}'(n):=
%$$
%contains exactly one root of $G_n$ for every $1\leq m\leq k_{1n}$. To
%complete the proof for $i=0$ note that on the random event $E_n$,
%|R_{0n}-r|
%=\left|\frac{{\log}|\xi_{\tau_n}|-\log_+|\xi_0|}{\tau_n}- \frac{{\log}|
%<n^{2\eps-1}
%by the definition of $E_n$.
%The case $i=d_n-1$ corresponding to the last segment of the majorant
%is similar.
Let us now consider the case $0<i<d_n-1$. On the event $E_n$, the
conditions of Lem\-ma~\ref{lemmain} are fulfilled for the polynomial
$G_n$ with $k=k_{in}$, $l=k_{(i+1)n}$ and $\delta=\zeta=\delta_n$;
see~(\ref{eqEn101}). The statement follows by Lemma~\ref{lemmain}.

It remains to prove that the sets $Z_{i,m}(n)$ are disjoint. It
suffices to show that on $E_n$ it holds that
$R_{(i+1)n}'-R_{in}'>3\delta_n$ for every $0\leq i<d_n$.
We have
%
%e56 #&#
%
\begin{equation}
\label{eqestdistRin}\qquad
(k_{(i+2)n}-k_{(i+1)n}) (R_{(i+1)n}-R_{in})=
S_{in}-R_{in}k_{(i+2)n}-{\log_+}|\xi_{k_{(i+2)n}}|.
\end{equation}
For $i\neq0, d_n-1$ it follows from (\ref{eqEn101}) that the
right-hand side can be estimated below by $n^{1/{\alpha}-\eps}$
on $E_n$. The required follows since $k_{(i+2)n}-k_{(i+1)n}\leq n$.
Using (\ref{eqestdistRin}) we obtain that for $i=0$ on the event
$E_n$ it holds that
\begin{eqnarray*}
\frac{k_{2n}-k_{1n}}{k_{2n}}(R_{1n}-R_{0n}) &=& \frac{{\log}|\xi_{\tau
_n}|}{\tau_n}-
\frac{{\log}|\xi_{k_{2n}}|}{k_{2n}}- {\log_+}|\xi_0| \biggl(\frac1{\tau_n}-
\frac
{1}{k_{2n}} \biggr)
\\
&>& n^{{1}/{\alpha}-1-2\eps},
\end{eqnarray*}
where the last inequality follows from (\ref{eqEn301}), (\ref
{eqEn601}). It follows that $R_{1n}-R_{0n}>n^{{1}/{\alpha
}-1-2\eps}$. Recalling (\ref{eqdiffRinRinprime}) we obtain
$R_{1n}'-R_{0n}'>3\delta_n$.
The case $i=d_n-1$ is similar.
%Thus, all roots of $G_n$ are localized and the proof is complete.
%The above arguments localize $k_{d_n n}'-k_{0n}'$ roots of $G_n$.
%Trivially, we have a root of multiplicity $k_{0n}'$ at zero. Since the
%total number of roots of $G_n$ is $k_{d_nn}'$, there are no other
%roots of $G_n$.
\end{pf}

Recall from (\ref{eqproofmain}) that we need to prove that $S_n\to
S$ in distribution as $n\to\infty$. Define a random variable $S_n^*$
which approximates $S_n$ by
\[
S_n^*=\frac1n \sum_{0\leq i<d_n-1}(k_{(i+1)n}-k_{in})
\bar f (b_nR_{in} ).
\]

%le4.9 #&#
%
\begin{lemma}
On the random event $E_n$ it holds that $|S_n^*-S_n|<\omega_f(n^{-\eps})$.
\end{lemma}
\begin{pf}
Assume that the event $E_n$ occurs. Take some $0\leq i<d_n$. Write
$\Delta_{in}=k_{(i+1)n}-k_{in}$. By Lemma~\ref{lemisolroots01},
the polynomial $G_n$ has a unique root, denoted by $z_{i,m}(n)$, in the
set $Z_{i,m}(n)$ for every $1\leq m\leq\Delta_{in}$. Denote by $\ZZZ
_{in}$ the finite set $\{z_{i,m}(n)\dvtx1\leq m\leq\Delta_{in}\}$.
Recall from (\ref{eqEn201}) that $\Delta_{in}>\sqrt n$. It follows
from the definition of the set $Z_{i,m}(n)$ that for every $1\leq m\leq
\Delta_{in}$,
\[
\biggl\llvert f\bigl({b_n\log}\bigl|z_{i,m}(n)\bigr|,\arg
z_{i,m}(n)\bigr)- \frac{\Delta_{in}}{2\pi}\int_{({\varphi_{in}+2\pi
m-\pi})/{\Delta
_{in}}}^{({\varphi_{in}+2\pi m+\pi})/{\Delta
_{in}}}f(b_nR_{in},
\varphi)\,d\varphi\biggr\rrvert
\]
is smaller than $\omega_f(n^{-\eps})$.
Note that for $i=0$ and $i=d_n-1$, we need to use (\ref
{eqdiffRinRinprime}) to prove this estimate.
Taking the sum over $1\leq m\leq\Delta_{in}$, we obtain
\[
\frac1n\biggl\llvert\sum_{z\in\ZZZ_{in}} f \bigl({b_n
\log}|z|, \arg z \bigr)- \Delta_{in}\bar f (b_nR_{in}
)\biggr\rrvert\leq\frac{\Delta_{in}}n \omega_f\bigl(n^{-\eps}
\bigr).
\]
Taking the sum over $0\leq i<d_n$, we obtain what is required.
\end{pf}
%
%le4.10 #&#
%
\begin{lemma}\label{lemPsimajGn}
We have $S_n^*\to S$ in distribution as $n\to\infty$.
%With the notation of Section~\ref{subsecproofmainnot} we have
%$$
%S_n^*:=\frac1n \sum_{0\leq i<d_n-1}(k_{(i+1)n}-k_{in}) \bar%f
%f(R_i)=S.
%$$
\end{lemma}
\begin{pf}
By\vspace*{1pt} Proposition~\ref{propresnick} the point process $\rho_n=\sum
_{k=0}^n\delta(\frac kn, \frac{{\log}|\xi_k|}{a_n})$ converges weakly
to $\rho$. By Lemma~\ref{lemPsicont} and Proposition \ref
{propcontmapping} (which is applicable since $\P[\rho\in\MMM_0]=1$ for
$\alpha\in(0,1)$) we have that $\Psi_0(\rho_n)$
converges weakly (as a random probability measure on $\R$) to $\Psi
_0(\rho)$. It follows that $\int_{\R}\bar f \,d \Psi_0(\rho_n)$
converges in distribution to $\int_{\R}\bar f \,d \Psi_0(\rho)$,
which is exactly what is stated in the lemma.
%$$
%n \bar f\left(R_{i,n}\right)
%$$
\end{pf}
%
%Take some $0\leq i<d_n$ and denote by $\ZZZ_{in}$ the set of roots of
%the random polynomial $G_n$ which satisfy $|b_n{\log}|z|- R_{in}|\leq
%1/n^2$. By \eqref{eqEn201} we have $k_{(i+1)n}-k_{in}>2\sqrt n$.
%Using Lemma~\ref{lemisolroots01} and comparing Riemann integrals to
%Riemann sums we obtain that on the event $E_n$ for every $0\leq i<d_n$,
%$$
%$$

The proof of Theorem~\ref{theocomplex} in the case $\alpha\in(0,1)$
can be completed as follows.
By Lemma~\ref{lemgoodeventconvdistr} with $S_n(\kappa)=S_n^*$ and
$S(\kappa)=S$, we obtain $S_n\to S$ in distribution as $n\to\infty$.
This proves (\ref{eqproofmain}).

The following explains how to get rid of the assumption $\xi_0\neq0$
a.s. Let $\P[\xi_0\neq0]$ be strictly positive. Denote the first
(resp., last) nonzero coefficient of $G_n$ by $\xi_{l_n}$ (resp.,
$\xi_{n-m_n}$). For fixed $l,m\in\N_0$, consider the conditional
distribution $\P_{l,m}^{n}$ of the random variables $\xi_k$, $l\leq
k\leq n-m$, given that\vspace*{1pt} $l_n=l$, $m_n=m$. Under $\P_{l,m}^{n}$, these
variables\vspace*{1pt} are independent and, apart from the first and the last
variable, identically distributed. It is easily seen that the above
proof applies to the polynomial $\sum_{k=l}^{n-m}\xi_k z^k$ under $\P
_{l,m}^{n}$. Since this holds for all $l,m\in\N_0$, the proof is complete.

%s5 #&#
\section{\texorpdfstring{Proof of Theorem \protect\ref{theocomplexalpha0}}{Proof of Theorem 1.5}}
Recall that $\tau_n\in\{0,\ldots,n\}$ is such that
$
M_n:=\max_{k=0,\ldots,n} {\log}|\xi_k|={\log}|\xi_{\tau_n}|.
$
Intuitively, under the slow variation condition (\ref
{eqtailalpha0}), the maximum $M_n$ is, with probability close to $1$,
much larger than all the other terms ${\log}|\xi_k|$, $1\leq k\leq n$.
The majorant of the set $\{(j,{\log}|\xi_j|)\dvtx j=0,\ldots,n\}$
consists, with high probability, of two segments joining the endpoints
$(0, \log_+ |\xi_0|)$ and $(n,\log_+ |\xi_n|)$ to the maximum
$(\tau_n, {\log}|\xi_{\tau_n}|)$. The roots of $G_n$ group around two
circles corresponding to these segments. Our aim is to make this
precise. Let the index $k$ be always restricted to $0\leq k\leq n$. We
may always assume that the index $\tau_n$ is defined uniquely, since
this event has probability converging to $1$ as $n\to\infty$;
see \cite{darling}.
%
%le5.1 #&#
%
\begin{lemma}\label{lem0En}
For $\kappa\in(0,1/2)$, $A>0$ define a random event $E_n=\bigcap
_{i=1}^4E_n^i$, where
\begin{eqnarray*}
%E_n^0&=\{\},\\
E_n^1&=& \biggl\{\min_{k\neq0,\tau_n} \biggl(
\frac{M_n}{\tau_n}-\frac{{\log}|\xi_k|} {k} \biggr) > n^{2A} \biggr\},
\\
E_n^2&=& \biggl\{\min_{k\neq\tau_n,n} \biggl(
\frac{M_n}{\tau_n}-\frac
{{\log}|\xi_k|} {n-k} \biggr) > n^{2A} \biggr\},
\\
E_n^3&=&\bigl\{\kappa n<\tau_n<(1-\kappa)n
\bigr\},
\\
E_n^4&=&\bigl\{\bigl|{\log}|\xi_0|\bigr|<n^{A},
M_n>n^{2A+1}, \bigl|{\log}|\xi_n|\bigr|<n^{A}
\bigr\}. %E_n^4&=&\{\tau_n<\}.
\end{eqnarray*}
Then, for every $A>0$, $\lim_{\kappa\downarrow0}\liminf_{n\to
\infty}\P[E_n]=1$.
%For every $a>0$ it holds that
%$$
%=
%=0.
%$$
\end{lemma}
\begin{pf}
By symmetry, $\tau_n/n$ converges as $n\to\infty$ to the uniform
distribution, which implies that $\lim_{\kappa\downarrow0}\liminf_{n\to
\infty}\P[E_n^3]=1$.
By \cite{darling}, Theorem 3.2, the slow variation condition (\ref
{eqtailalpha0}) implies that
\[
\frac1{M_n}\mathop{\max_{0\leq k\leq n}}_{k\neq\tau_n} {\log}|
\xi_k| \toprobab0.
\]
%
%converges in probability to $0$ as $n\to\infty$.
It follows that
%
%e57 #&#
%
\begin{equation}
\label{eqproof01} \P\biggl[\mathop{\max_{\kappa n\leq k<n}}_{k\neq\tau
_n }
\frac
{{\log}|\xi_k|}{k}>\frac{M_n} {2n} \biggr] \leq\P\biggl[\mathop{
\max_{0\leq k\leq n}}_{k\neq\tau_n} {\log}|\xi_k|>\frac{\kappa}{2}
M_n \biggr] \ton0. % n\to\infty.
\end{equation}
Put $c_n=\inf\{s\dvtx\bar F(s)\leq1/(\sqrt{\kappa} n)\}$. Then
$\bar
F(c_n)\sim1/(\sqrt{\kappa} n)$ by \cite{resnickbook}, pages 15 and~16,
and $\lim_{n\to\infty}c_n/n=\infty$.
Recall the Potter bound for slowly varying functions: for every $\delta
>0$, we have $\bar F(y)/\bar F(x)< 2(x/y)^{\delta}$, provided that
$x>y$ are sufficiently large; see \cite{binghambook}, Theorem 1.5.6.
We have
%Put
%$$
%c_n
%=\inf\left\{s\dvtx\bar F(s)\leq\frac{1}{\sqrt{\kappa} n}\right\}.
%$$
%
%e58 #&#
%
\begin{eqnarray}
\label{eqproof02} \P\biggl[\max_{1\leq k\leq\kappa n}\frac{{\log}|\xi
_k|}{k} >
\frac
{M_n}{2n} \biggr] &\leq& \sum_{1\leq k\leq\kappa n} \bar F
\biggl(\frac{k}{2n} c_n \biggr)+\P[M_n<c_n
]
\nonumber\\
&<& \frac{3}{\sqrt{\kappa} n}\sum_{1\leq k\leq\kappa n} \biggl(
\frac
{2n}{k} \biggr)^{1/4}+ \biggl(1-\frac{1}{2\sqrt{\kappa} n}
\biggr)^{n+1}
\\
&<& C\bigl(\kappa^{1/4}+e^{-1/(2\sqrt{\kappa})}\bigr).
\nonumber
\end{eqnarray}
Since $\bar F$ decays more slowly than any negative power of $n$,
%
%e59 #&#
%
\begin{equation}
\label{eqproof03}\qquad \P\biggl[\frac{M_n}{2n}>n^{2A} \biggr]=1-
\bigl(1-\bar F\bigl(2n^{2A+1}\bigr)\bigr)^{n+1}>1- \biggl(1-
\frac1{n^{2}} \biggr)^{n+1}\ton1. % n\to\infty.
\end{equation}
Putting (\ref{eqproof01}), (\ref{eqproof02}) and (\ref
{eqproof03}) together and letting $\kappa\downarrow0$, we obtain\break
$\lim_{n\to\infty}\P[E_n^1]=1$. By symmetry, we also have $\lim_{n\to
\infty}\P[E_n^2]=1$. From (\ref{eqproof03}) it also follows
that $\lim_{n\to\infty}\P[E_n^4]=1$.
%$$
%{aM_n} n\right]
%<
%+\P\left[M_n<\bar F^{-1}\left(\frac{1}{\kappa n}\right)\right].
%$$
\end{pf}
\begin{pf*}{Proof of Theorem~\ref{theocomplexalpha0}}
In the sequel, we always suppose that the event $E_n$ occurs. The roots
of the equation $\xi_{\tau_n} z^{\tau_n}+\xi_0=0$, denoted by
$w_{1n},\ldots,w_{\tau_n n}$, satisfy
\[
|w_{kn}|=\bigl(|\xi_0|/|\xi_{\tau_n}|\bigr)^{1/\tau_n}=e^{({\log}|\xi
_0|-M_n)/\tau_n}<
e^{-n^A},\qquad 1\leq k\leq\tau_n.
\]
Similarly, the roots of the equation $\xi_{n} z^{n-\tau_n}+\xi_{\tau
_n}=0$, denoted by $w_{(\tau_n+1)n},\ldots,\break w_{nn}$, satisfy
\[
|w_{kn}|=\bigl(|\xi_{\tau_n}|/|\xi_{n}|\bigr)^{1/(n-\tau_n)}=e^{(M_n-{\log}
|\xi_n|)/(n-\tau_n)}>
e^{n^A},\qquad \tau_n<k\leq n.
\]
Choose $s,r\in\R$ so that $s={\log}|\xi_0|$ and $s-r\tau_n={\log}|\xi_{\tau
_n}|=M_n$. %In the sequel, we always assume that the event $E_n$
%occurs.
To apply Lem\-ma~\ref{lemmain} with $k=0$, $l=\tau_n$ we need to
estimate\vadjust{\goodbreak} $h:=\min_{k\neq0,\tau_n}(s-rk-{\log}|\xi_k|)$.
We have, by definition of $E_n$,
\[
%=\min_{{1\leq k\leq n\\k\neq\tau_n}} \left(\frac{M_n}{\tau_n}-\frac{
%{\log}|\xi_k|}{k}+s\left(\frac1k -\frac1 {\tau_n}\right) \right)
%>n^{\frac32 A}.
\min_{k\neq0,\tau_n}
\frac{s-rk-{\log}|\xi_k|}k =\min_{k\neq0,\tau_n} \biggl(\frac{M_n}{\tau_n}-
\frac{{\log}|\xi_k|}{k}+s \biggl(\frac1k -\frac1 {\tau_n} \biggr) \biggr)
>n^{3A/2}.
\]
Hence, $h>n^{3A/2}$.
%It follows that on the event $E_n$ the segment joining $(0, {\log}|
%Note that $h^*<h$, where $h$ is as in \eqref{eqmodulusxikl}.
It follows that on the event $E_n$ the conditions of Lemma~\ref
{lemmain} are fulfilled for $k=0$, $l=\tau_n$ and $\delta=\zeta
=e^{-n^A}$. Then, for every $1\leq k\leq\tau_n$, the set
\[
\bigl\{z\in\C\dvtx\bigl\llvert{\log}|z|-r\bigr\rrvert\leq e^{-n^A},
\llvert{\arg z}-\arg w_{kn}\rrvert\leq e^{-n^A} \bigr\}
\]
contains exactly one root, say $z_{kn}$, of the polynomial $G_n$. It
follows that
%Using the inequality $|z_1-z_2|\leq|z_1| |\arg z_1-\arg z_2|+
%|z_2|-|z_1|$, we obtain that
%
\[
|z_{kn}-w_{kn}|< 10\delta e^r= 10
e^{-n^A}|w_{kn}|,\qquad 1\leq k\leq\tau_n. %e^r \delta+ e^r(e^{\delta}-1)
\]
By symmetry, a similar inequality holds for $\tau_n<k\leq n$.
\end{pf*}
%
%The proof is complete.

%s6 #&#
\section{\texorpdfstring{Proofs of Theorems \protect\ref{theoreal} and \protect\ref{theorealalpha0}}
{Proofs of Theorems 1.10 and 1.14}}\label{secproofreal}
%s6.1 #&#
\subsection{Limiting point processes}\label{secproofrealdefproc}
First of all, we describe the limiting point processes $\Upsilon_{\alpha
,c}$ and $\Upsilon_{\alpha,c,p}^{\pm}$. Let $\rho$ be a
Poisson point process on $[0,1]\times(0,\infty)$ with intensity
$\alpha v^{-(\alpha+1)}\,du\,dv$ and majorant $\CCC_{\rho}$ as in
Section~\ref{subseccomplexroots}. Recall that the vertices of the
majorant $\CCC_{\rho}$ are denoted by $(X_k,Y_k)$. For $\alpha\geq
1$ the index $k$ ranges in $\Z$, whereas for $\alpha\in(0,1)$ we
have $p'\leq k\leq p''$ and $(X_{p'},Y_{p'})=(0,0)$,
$(X_{p''},Y_{p''})=(1,0)$. Let $\sigma_{k},\pi_{k}$ be independent $\{
-1,1\}$-valued random variables [attached to the \textit{vertices}
$(X_k,Y_k)$ of $\CCC_{\rho}$ except for the boundary vertices $(0,0)$
and $(1,0)$ in the case $\alpha\in(0,1)$] such that
\[
\P[\sigma_{k}=1]=c,\qquad \P[\pi_{k}=1]=1/2. % p'<k<p''.
\]
In the case $\alpha\in(0,1)$, we have to add the following boundary
conditions:
\begin{longlist}[(3)]
\item[(1)] $\pi_{p'}=1$;
\item[(2)] $\pi_{p''}=1$ in the definition of
$\Upsilon_{\alpha,c,p}^{+}$ and $\pi_{p''}=-1$ in the definition of~$\Upsilon_{\alpha,c,p}^{-}$;
\item[(3)] $\P[\sigma_{p'}=1]=\P[\sigma_{p''}=1]=p$. %\P[\xi_0>0]$.
\end{longlist}
Define random variables $\eps_k^+$ and $\eps_k^-$ attached to the
\textit{linearity intervals} $[X_k,\break X_{k+1}]$ of the majorant $\CCC_{\rho
}$ by
%
%e60 #&#
%
\begin{equation}
\eps_k^+=\ind_{\{\sigma_{k}\neq\sigma_{k+1}\}},\qquad \eps_k^-=
\ind_{\{\sigma_{k}\pi_k\neq\sigma_{k+1}\pi_{k+1}\}}.
\end{equation}
With this notation, the limiting point processes $\Upsilon_{\alpha
,c}$ and $\Upsilon_{\alpha,c,p}^{\pm}$ are defined by
%
%e61 #&#
%
\begin{equation}
\label{eqPireal} \Upsilon_{\alpha,c(,p)}^{(\pm)} = \sum
_{k}\eps_k^+\delta\bigl(e^{R_k}\bigr)+
\sum_{k}\eps_k^-\delta
\bigl(-e^{R_k}\bigr),
\end{equation}
where the sum is over all linearity intervals of the majorant $\CCC
_{\rho}$, and $R_k$ is the negative of the slope of the $k$th segment
of $\CCC_{\rho}$ as in (\ref{eqCCCrho}). We proceed to the proof
of Theorem~\ref{theoreal}.

%s6.2 #&#
\subsection{\texorpdfstring{Proof in the case $\alpha\geq1$}{Proof in the case alpha >= 1}}
We will show that the following weak convergence of point processes on
$E=\R\times\{-1,1\}$ holds true:
%
%e62 #&#
%
\begin{equation}
\label{eqmainrealrestated} \sum_{z\in\RRR_n}
\delta\bigl({b_n\log}|z|, \sgn z\bigr)\toweak\sum_{k}
\eps_k^+ \delta(R_k,1)+\sum_{k}
\eps_k^-\delta(R_k,-1),
\end{equation}
where the sum on the right-hand side is over all linearity intervals of
the majorant $\CCC_{\rho}$.
To see that (\ref{eqmainrealrestated}) implies Theorem \ref
{theoreal} for $\alpha\geq1$ note that the mapping $F\dvtx E\to\R\bsl\{
0\}$ given by $F(r,\sigma)=\sigma e^r$ is continuous and proper
(preimages of compact sets are compact).
By \cite{resnickbook}, Proposition 3.18, it induces a vaguely
continuous mapping between the
spaces of locally finite counting measures on $E$ and $\R\bsl\{0\}$.
By Proposition~\ref{propcontmapping} we may apply this mapping to
the both sides of (\ref{eqmainrealrestated}), which implies the
statement of Theorem~\ref{theoreal} for $\alpha\geq1$.
Denote by $\RRR_n^+$ (resp., $\RRR_n^-$) the set of positive (resp.,
negative) real roots of $G_n$, counted with multiplicities.
Let $f^+,f^-\dvtx\R\to[0,\infty)$ be two continuous functions supported
on an interval $[-A,A]$. Define random variables $S_n$ and $S$ by
%
%e63 #&#
%e64 #&#
%
\begin{eqnarray}
\label{eqSnreal} S_n&=&\sum_{z\in\RRR_n^+}
f^+(b_n\log z)+\sum_{z\in\RRR_n^-}
f^-\bigl({b_n\log}|z|\bigr),
\\
\label{eqSreal} S&=&\sum_{k} \eps_k^+
f^+(R_k)+\sum_{k} \eps_k^-
f^-(R_k),
\end{eqnarray}
where the sum in (\ref{eqSreal}) is over all linearity intervals of
$\CCC_{\rho}$.
To prove (\ref{eqmainrealrestated}) it suffices to show that
$S_n\to S$ in distribution as $n\to\infty$. In fact, we may even
suppose additionally that $f^+$ and $f^-$ are Lipschitz, that is
$|f^{\pm}(z_1)-f^{\pm}(z_2)|<L|z_1-z_2|$ for some $L>0$ and all
$z_1,z_2\in\R$.
The first step is to localize the real roots of $G_n$ under some
``good'' event. We use the same notation as in Section \ref
{subsecproofmainnot}. Take $\kappa\in(0,1/2)$ and recall that the
random indices $q_n'$ and $q_n''$ have been defined in (\ref{eqqn}).
Define a random event $E_n$ as in Lemma~\ref{lemEngeq1}.
Additionally, we will need another ``good'' event $F_n$. The next lemma
states that it has probability close to $1$.
%
%le6.1 #&#
%
\begin{lemma}\label{lemFn}
Consider a random event
$
F_n=\{ b_n R_{q_n'n}<-2A\}\cap\{b_n\* R_{(q_n''-1)n}>2A\}.
$
Then, $\lim_{\kappa\downarrow0} \liminf_{n\to\infty} \P[F_n]=1$.
\end{lemma}
\begin{pf}
Recall from Section~\ref{secmajorantconv} that $\MMM$ is the space
of locally finite counting measures on $[0,1]\times(0,\infty]$ which
do not charge the set $[0,1]\times\{\infty\}$. Given $\mu\in\MMM$
we denote by $[x_{q'},x_{q'+1}]$ the unique linearity interval of the
majorant $\CCC_{\mu}$ such that $x_{q'}\leq\kappa< x_{q'+1}$.
Denote by $r_{q'}$ the negative of the slope of the corresponding
segment of $\CCC_{\mu}$. Define a map $T_{\kappa}\dvtx\MMM\to\R$ by
$T_{\kappa}(\mu)=r_{q'}$. Then, the same argument as in Lemma \ref
{lemPsietacont} shows that $T_{\kappa}$ continuous on $\MMM_1$;
see~(\ref{eqlemcont2}). Applying Proposition~\ref{propresnick}
together with Proposition~\ref{propcontmapping} and noting that
$T_{\kappa}(\rho_n)=b_n R_{q_n'n}$ we obtain that for every $\kappa
>0$, $b_n R_{q_n'n}\to T_{\kappa}(\rho)$ in distribution as $n\to
\infty$. By Proposition~\ref{propnumberverticesmajorant} we have
$T_{\kappa}(\rho)\to-\infty$ a.s. as $\kappa\downarrow0$. It
follows easily that $\lim_{\kappa\downarrow0} \liminf_{n\to\infty
} \P[b_n R_{q_n'n}<-2A]=1$. The statement of the lemma follows by symmetry.
\end{pf}

In the next lemma we will localize, under the event $E_n\cap F_n$,
those real roots of $G_n$ which are contained in $[-A,A]$. Recall that
the vertices of the majorant of $G_n$ are denoted (from left to right)
by $(k_{in}, {\log_+}|\xi_{k_{in}}|)$, where $0\leq i\leq d_n$ and
$k_{0n}=0$, $k_{d_nn}=n$. We already know that any linearity interval
$[k_{in},k_{(i+1)n}]$ of the majorant corresponds to a ``circle'' of
\textit{complex} roots of $G_n$ located approximately at the same
positions as the nonzero roots of the polynomial $\xi
_{k_{in}}z^{k_{in}}+\xi_{k_{(i+1)n}}z^{k_{(i+1)n}}$.
In order to localize the \textit{real} roots of $G_n$ we have to keep
track of two things: the signs of the coefficients $\xi_{k_{in}}, \xi
_{k_{{(i+1)n}}}$ and the parities of the indices $k_{in},k_{(i+1)n}$. Write
%
%e65 #&#
%e66 #&#
%
\begin{eqnarray}
\label{eqepsinplusreal} \eps_{in}^+&=&\ind\bigl\{\sgn(
\xi_{k_{in}})\neq\sgn(\xi_{k_{(i+1)n}})\bigr\},
\\
\label{eqepsinminusreal} \eps_{in}^-&=&\ind\bigl\{(-1)^{k_{in}}
\sgn(\xi_{k_{in}})\neq(-1)^{k_{(i+1)n}}\sgn(\xi_{k_{(i+1)n}})\bigr
\}.
\end{eqnarray}
The next lemma shows that $\eps_{in}^+$ (resp., $\eps_{in}^-$) is the
indicator of the presence of a real root of $G_{n}$ near $e^{R_{in}}$
(resp., $-e^{R_{in}}$).
%
%le6.2 #&#
%
\begin{lemma}\label{lemlocalzerogeq1real}
On the random event $E_n$ the following holds: for every $q_n'\leq
i<q_n''$ such that $\eps_{in}^+=1$ (resp., $\eps_{in}^-=1$) there is
exactly one positive (resp., negative) real root of $G_n$ satisfying
$|{\log}|z|-R_{in}|\leq\exp(-n^{{1}/{\alpha}-2\eps})$.
Moreover, if additionally $F_n$ occurs, then all real roots of $G_n$
satisfying ${b_n\log}|z|\in[-A,A]$ are among those described above.
\end{lemma}
\begin{pf}
We will use the notation of Lemma~\ref{lemlocalzerogeq1}. Recall
that on the event $E_n$ for every $q_n'\leq i< q_n''$ and every $1\leq
m\leq k_{(i+1)n}-k_{in}$ there is a unique complex root of $G_n$,
denoted by $z_{i,m}(n)$, in the set $Z_{i,m}(n)$. Let $\eps_{in}^+=1$
for some \mbox{$q_n'\leq i<q_n''$}. Then, $\varphi_{in}=0$ in Lemma \ref
{lemlocalzerogeq1}. Setting $m=k_{(i+1)n}-k_{in}$ we have that
$z:=z_{i,m}(n)$ satisfies $|{\log}|z|-R_{in}|<\delta_n$ and $|{\arg
z}|<\delta_n$. Since the coefficients of $G_n$ are real, the root $z$
must in fact be real (and positive). Indeed, otherwise, we would have a
pair complex conjugate roots (rather than a single root) in the set
$Z_{i,m}(n)$. Similarly, if $\eps_{in}^-=1$ for some $q_n'\leq
i<q_n''$, then we have a real negative root of the form $z_{i,m}(n)$
for a suitable $m$. By Lemma~\ref{lemlocalzerogeq1} all real roots
in the set $R_{q_n'n}-\delta_n\leq{\log}|z|\leq R_{(q_n''-1)n}+\delta_n$
are of the above form. To complete the proof note that this set
contains the set $-A\leq {b_n\log}|z|\leq A$ on the event $F_n$.
\end{pf}
The random variables $S_n$ and $S$ will be approximated by the random
variables $S_n(\kappa)$ and $S(\kappa)$, defined by
%
%e67 #&#
%e68 #&#
%
\begin{eqnarray}
\label{eqSnkappareal} S_n(\kappa) &=& \sum
_{q_n'<i<q_n''-1} \bigl(\eps_{in}^+ f^+(b_nR_{in})+
\eps_{in}^- f^-(b_nR_{in})\bigr),
\\
\label{eqSkappareal} S(\kappa) &=& \sum_{q'<i<q''-1} \bigl(
\eps_i^+ f^+(R_i)+\eps_i^-
f^-(R_i)\bigr).
\end{eqnarray}

%le6.3 #&#
%
\begin{lemma}
On the random event $E_n\cap F_n$, we have $|S_n-S_n(\kappa)|<1/n$.
\end{lemma}
\begin{pf}
Recall that $f^+$ and $f^-$ are functions supported on $[-A,A]$ with
Lipschitz constant at most $L$. By Lemma \ref
{lemlocalzerogeq1real} and the definition of $F_n$, we have, on
$E_n\cap F_n$,
\[
\Biggl\llvert\sum_{z\in\RRR_n^+} f^+(b_n\log
z)-\sum_{i=0}^{d_n-1} \eps_{in}^+
f^+(b_nR_{in})\Biggr\rrvert\leq Ld_n
b_n \exp\bigl(-n^{{1}/{\alpha
}-2\eps}\bigr) \leq\frac1{2n}.
\]
A similar inequality holds for the negative roots, and the statement follows.
\end{pf}
The next proposition determines the limiting structure of the
coefficients of $G_n$ together with attached signs and parities. Let
$\tilde{\MMM}$ be the space of locally finite counting measures on
$[0,1]\times(0,\infty]\times\{-1,1\}^2$ which do not charge the set
$[0,1]\times\{\infty\}\times\{-1,1\}^2$. We endow $\tilde\MMM$
with the topology of vague convergence. Every element $\tilde\mu\in
\tilde\MMM$ can be written in the form $\tilde\mu=\sum_{i}\delta
(u_i,v_i, \varsigma_i,\varpi_i)$, where $\mu=\sum_{i}\delta
(u_i,v_i)\in\MMM$ is the projection of $\tilde\mu$ on $\MMM$ and
$(\varsigma_i, \varpi_i)\in\{-1,1\}^2$ is considered as a mark
attached to the point $(u_i,v_i)$. In the marks $(\varsigma_i, \varpi
_i)$ we will record the signs of the coefficients of $G_n$ and the
parities of the corresponding indices.
%
%pr6.4 #&#
%
\begin{proposition}\label{propresnickreal}
Let $\xi_0,\xi_1,\ldots$ be i.i.d. random variables
satisfying (\ref{eqtail1}) and~(\ref{eqtailreal}). Then the
following convergence holds weakly on the space $\tilde\MMM$:
%
%e69 #&#
%
\begin{equation}
\label{eqtilderhontilderho} \tilde\rho_n:=\sum
_{k=0}^{n} \delta\biggl(\frac kn,
\frac{{\log}
|\xi_k|}{a_n}, \sgn\xi_k, (-1)^{k} \biggr) \toweak
\sum_{i=1}^{\infty}\delta(U_i,V_i,
\varsigma_i,\varpi_i) =: \tilde\rho.\hspace*{-20pt}
\end{equation}
Here,\vspace*{1pt} $\rho=\sum_{i=1}^{\infty} \delta(U_i,V_i)$ is a Poisson point
process on $[0,1]\times(0,\infty)$ with intensity $\alpha v^{-(\alpha
+1)}\,du\,dv$ and independently, $\varsigma_i, \varpi_i$ are $\{-1,1\}
$-valued random variables with $\P[\varsigma_{i}=1]=c$ and $\P
[\varpi_{i}=1]=1/2$. Terms with ${\log}|\xi_k|\leq0$ are ignored.
\end{proposition}
\begin{pf}
Write $\xi_k^+=\xi_k \ind_{\xi_k>0}$ and $\xi_k^-=|\xi_k| \ind_{\xi
_k\leq0}$.
Note that by (\ref{eqtail1}), (\ref{eqdefan}) and (\ref{eqtailreal}),
\[
\P\biggl[\frac{\log\xi_k^+}{a_n}> t \biggr] \sim\frac c
{nt^{\alpha}},\qquad
\P
\biggl[\frac{\log\xi_k^-}{a_n}> t \biggr] \sim\frac{1-c} {
nt^{\alpha}}, \qquad n\to\infty.
\]
Fix some $(\varsigma,\varpi)\in\{-1,1\}^2$. We will consider only
coefficients $\xi_k$ with sign $\varsigma$ and parity $\varpi$. By
Proposition~\ref{propresnick} the point process
\[
\tilde\rho_n(\varsigma, \varpi):=\sum_{k=0}^{n}
\delta\biggl(\frac{k}n, \frac{{\log}|\xi_{k}|}{a_{n}} \biggr)\ind\bigl\{
\sgn(
\xi_k)=\varsigma, (-1)^k=\varpi\bigr\}
\]
converges weakly to the Poisson point process with intensity $(\alpha
/2) c v^{-(\alpha+1)}\,du\,dv$ if $\varsigma=1$ and $(\alpha/2) (1-c)
v^{-(\alpha+1)}\,du\,dv$ if $\varsigma=-1$. Taking the union over all $4$
choices of $(\varsigma,\varpi)$, we obtain the statement.
\end{pf}

In order to pass from the convergence of the coefficients to the
convergence of the point process of real roots we need a continuity
argument. Consider $\tilde\mu\in\tilde\MMM$ with a projection $\mu
\in\MMM$. We denote the vertices of the majorant of $\mu$ counted
from left to right by $(x_k, y_k)$. Denote by $r_k$ the negative of the
slope of the majorant of $\mu$ on the interval $[x_k,x_{k+1}]$. Let
$\kappa\in(0,1/2)$ be fixed and define indices $q'$ and $q''$ by the
conditions $x_{q'}\leq\kappa<x_{q'+1}$ and $x_{q''-1}<1-\kappa\leq
x_{q''}$. For $q'< k< q''$ we denote by $(\sigma_k, \pi_k)\in\{-1,1\}
^2$ the mark attached to the vertex $(x_k,y_k)$. Let $\tilde\MMM_1$
be the set of all $\tilde\mu\in\tilde\MMM$ such that $\mu\in\MMM_1$,
where $\MMM_1\subset\MMM$ is defined as in Section~\ref
{secmajorantconv}. Let $\PPP$ be the space of locally finite
counting measures on $\R$ endowed with the topology of vague
convergence. Define a map $\Phi_{1}\dvtx\tilde\MMM\to\PPP\times
\PPP
$ by
\[
\Phi_{1}(\tilde\mu) = \biggl( \sum_{q'<k<q''-1}
\ind_{\sigma_k\neq\sigma_{k+1}}\delta(r_k), \sum_{q'<k<q''-1}
\ind_{\sigma_k\pi_k\neq\sigma_{k+1}\pi
_{k+1}}\delta(r_k) \biggr).
\]

%if $\mu$ does not charge $\infty$ and by $\Phi_{\kappa}(\tilde\mu)=0$
%otherwise.
%
%le6.5 #&#
%
\begin{lemma}\label{lemPsicontrealgeq1}
The map $\Phi_{1}$ is continuous on $\tilde\MMM_1$.
\end{lemma}
\begin{pf}
Let $\{\tilde\mu_n\}_{n\in\N}\subset\tilde\MMM$ be a sequence
converging vaguely to $\tilde\mu\in\tilde\MMM_1$. This implies
the vague convergence of the corresponding projections: $\mu_n\to\mu
\in\MMM_1$. Arguing as in the proof of Lemma~\ref{lemPsietacont}
(and using the same notation) we arrive at the following conclusions.
There exist points $(x_{kn},y_{kn})$, $q'\leq k\leq q''$, which are
vertices of the majorant of $\mu_n$, such that $(x_{kn},y_{kn})\to
(x_k,y_k)$ as $n\to\infty$. Further, $x_{q'n}< \kappa<x_{(q'+1)n}$
and $x_{(q''-1)n}<1-\kappa<x_{q''n}$ for sufficiently large $n$. Also,
with the same notation as in (\ref{eqlemcont2}), $r_{kn}\to r_k$ as
$n\to\infty$.
Finally, $\tilde\mu_n\to\tilde\mu$ implies that for sufficiently
large $n$ the mark $(\sigma_{kn}, \pi_{kn})$ attached to
$(x_{kn},y_{kn})$ is the same as the mark $(\sigma_k,\pi_k)$ attached
to $(x_k,y_k)$, for all $q'\leq k\leq q''$.
This implies that $\Phi_{1}(\tilde\mu_n)\to\Phi_{1}(\tilde\mu)$
as $n\to\infty$, whence the continuity.
\end{pf}
%
%Define random variable $S(\kappa)$ by $S(\kappa)=\sum_{i=q'}^{q''-1} (
%S_n(\kappa)=\sum_{i=q_n'}^{q_n''-1} (\eps_{in}^+ f^+(b_nR_{in})+
%
%S(\kappa)=\sum_{i=q'}^{q''-1} (\eps_i ^+ f^+(R_i)+\eps_i^- f^-(R_i)).
%
%le6.6 #&#
%
\begin{lemma}
We have $S_n(\kappa)\to S(\kappa)$ in distribution as $n\to\infty$.
%The following distributional convergence holds:
%$$
%S_n(\kappa):=\sum_{i=q_n'}^{q_n''-1} (\eps_{in}^+ f^+(b_nR_{in})+
%$$
\end{lemma}
\begin{pf}
By Proposition~\ref{propresnickreal} we have $\tilde\rho_n\to
\tilde\rho$ weakly on $\tilde\MMM$. Define a map $I\dvtx\PPP\times
\PPP\to\R$ by $I(\nu^+,\nu^-)=\int_{\R} f^+ \,d\nu^++ \int_{\R
}f^- \,d\nu^-$. Clearly, $I$ is continuous on $\MMM_1$. By Lemma \ref
{lemPsicontrealgeq1} the map $I\circ\Phi_{1}\dvtx\tilde\MMM\to\R$
is continuous.
By Proposition~\ref{propcontmapping} (which is applicable since $\P
[\tilde\rho\in\tilde\MMM_1]=1$ for $\alpha\geq1$) we have that
$I(\Phi_{1}(\tilde\rho_n))\to I(\Phi_{1}(\tilde\rho))$ in
distribution. This is exactly what is stated in the lemma.
%weakly on $\PPP\times\PPP$. Taking the integrals of $f^+$ and $f^-$
%over the components of $\Phi_{\kappa}(\tilde\rho_n)$ and $\Phi_{
\end{pf}

The proof of Theorem~\ref{theoreal} in the case $\alpha\geq1$ can
be completed as follows. Trivially, we have $S(\kappa)\to S$ a.s. as
$\kappa\downarrow0$. All the other assumptions of Lemma~\ref
{lemgoodeventconvdistr} have been verified above. Applying
Lemma~\ref{lemgoodeventconvdistr} we obtain $S_n\to S$ in
distribution as $n\to\infty$.

%s6.3 #&#
\subsection{\texorpdfstring{Proof in the case $\alpha\in(0,1)$}{Proof in the case alpha in (0,1)}}
%The proof is similar to the previous case. We just indicate the
%differences as compared to the previous case.
We will show that the weak convergence of point processes in (\ref
{eqmainrealrestated}) holds, this time on the space $E=[-\infty
,+\infty]\times\{-1,1\}$ with the restriction that $n$ stays either
even or odd and $\eps_k^+, \eps_k^-$ on the right-hand side of (\ref
{eqmainrealrestated}) is defined accordingly to this choice (see the
boundary conditions in Section~\ref{secproofrealdefproc}). Let
$f^+,f^-\dvtx[-\infty,\infty] \to[0,\infty)$ be two continuous
functions such that $|f^{\pm}(z_1)-f^{\pm}(z_2)|<L|z_1-z_2|$ for all
$z_1,z_2\in\R$. With the same notation as in (\ref{eqSnreal})
and (\ref{eqSreal}) it suffices to prove that $S_n\to S$ in
distribution as $n\to\infty$.
The next lemma localizes all real roots of $G_n$ under a ``good'' event.
%
%le6.7 #&#
%
\begin{lemma}\label{lemlocalzero01real}
On the random event $E_n$ defined as in Lemma~\ref{lemEn01} the
following holds:
For every $0\leq i<d_n$ such that $\eps_{in}^+=1$ (resp., $\eps
_{in}^-=1$) there is exactly one positive (resp., negative) real root
$z$ of $G_n$ satisfying $|{\log}|z|-R_{in}'|\leq\exp(-n^{1/
{\alpha}-1-3\eps})$.
Moreover, there are no other real roots of $G_n$.
\end{lemma}
\begin{pf}
Follows from Lemma~\ref{lemisolroots01}; see the proof of
Lemma~\ref{lemlocalzerogeq1real}.
\end{pf}
Take $\kappa\in(0,1/2)$, and define random variables $S_n(\kappa)$
and $S(\kappa)$ as in (\ref{eqSnkappareal}) and (\ref
{eqSkappareal}), but with summation over $q_n'\leq k<q_n''$ and
$q'\leq k<q''$.
%and recall that the random indices $q_n'$ and $q_n''$ have been
%defined in \eqref{eqqn}. Define a random variable $S_n(\kappa)$ by
%S_n(\kappa)=\sum_{i=q_n'}^{q_n''-1} (\eps_{in}^+ f^+(b_nR_{in})+
%
%le6.8 #&#
%
\begin{lemma}
On the random event $E_n$, we have $|S_n-S_n(\kappa)|<1/\sqrt{n}$.
\end{lemma}
\begin{pf}
By Remark~\ref{remEn01} we have $q_n'=0$ and $q_n''=d_n$ on $E_n$.
The rest follows from Lemma~\ref{lemlocalzero01real}, the
Lipschitz property of $f^+$ and $f^-$ and (\ref{eqdiffRinRinprime}).
\end{pf}

Again, we need a continuity argument to transform the convergence of
the coefficients in Proposition~\ref{propresnickreal} into the
convergence of real roots. This time, we have to take care of the first
and the last coefficients of the random polynomial $G_n$. Write $\KKK
=\tilde{\MMM}\times\{-1,1\}^2$. Every element of $\KKK$ can be
written in the form $(\tilde\mu, \sigma',\sigma'')$, where $\tilde
\mu\in\tilde\MMM$ and $(\sigma',\sigma'')\in\{-1,1\}^2$. In
$\sigma'$ and $\sigma''$ we will record the signs of the first and
the last coefficients of $G_n$. As above, the vertices of the majorant
of $\mu$ counted from left to right are denoted by $(x_k, y_k)$ and
the indices $q'$ and $q''$ are defined by the conditions $x_{q'}\leq
\kappa<x_{q'+1}$ and $x_{q''-1}<1-\kappa\leq x_{q''}$. For $q'< k<
q''$ (note the strict inequalities) we denote by $(\sigma_k, \pi_k)\in\{
-1,1\}^2$ the mark attached to the vertex $(x_k,y_k)$. We will
need the following boundary conditions: Define $(\sigma_{q'},\pi
_{q'})=(\sigma',1)$ and put $(\sigma_{q''},\pi_{q''})=(\sigma'',1)$
(if we are proving the convergence of $\Upsilon_{2n}$) or $(\sigma
_{q''},\pi_{q''})=(\sigma'',-1)$ (if we are proving the convergence
of $\Upsilon_{2n+1}$). Let $\KKK_0$ be the set of all $(\tilde\mu,
\sigma',\sigma'')\in\KKK$ such that the projection $\mu$ of
$\tilde\mu$ satisfies $\mu\in\MMM_0$. Here, $\MMM_0\subset\MMM$
is defined as in Section~\ref{secmajorantconv}. Let $\QQQ$ be the
space of finite counting measures on $[-\infty,\infty]$ endowed with
the topology of weak convergence. Define a map $\Phi_{0}\dvtx\KKK\to
\QQQ\times\QQQ$ by
\[
\Phi_{0}\bigl(\tilde\mu, \sigma',\sigma''
\bigr) = \Biggl( \sum_{k=q'}^{q''-1}
\ind_{\sigma_k\neq\sigma_{k+1}}\delta(r_k), \sum_{k=q'}^{q''-1}
\ind_{\sigma_k\pi_k\neq\sigma_{k+1}\pi
_{k+1}}\delta(r_k) \Biggr).
\]

%if $\mu$ does not charge $\infty$ and by $\Phi_{\kappa}(\tilde\mu,
%
%le6.9 #&#
%
\begin{lemma}\label{lemPsicontreal01}
The map $\Phi_{0}$ is continuous on $\KKK_0$.
%$$
%=
%$$
\end{lemma}
\begin{pf}
Let $\{(\tilde\mu_n, \sigma_n',\sigma_n'')\}_{n\in\N}\subset\KKK
$ be a sequence converging vaguely to $(\tilde\mu, \sigma',\sigma'')\in
\KKK_0$. This implies that for sufficiently large $n$, $\sigma_n'=\sigma
'$ and $\sigma_n''=\sigma''$. Also, $\tilde\mu_n\to
\tilde\mu$ vaguely. Consequently, we have the vague convergence of
the corresponding projections: $\mu_n\to\mu$. As in the proof of
Lemma~\ref{lemPsicont} we obtain the following results. There exist
points $(x_{kn},y_{kn})$, $q'<k<q''$, which are vertices of the
majorant of $\mu_n$, such that $(x_{kn},y_{kn})\to(x_k,y_k)$ as $n\to
\infty$. Also, $x_{q'n}< \kappa<x_{(q'+1)n}$ and
$x_{(q''-1)n}<1-\kappa<x_{q''n}$ for sufficiently large $n$.
Furthermore, with the same notation as in (\ref{eqlemcont2}),
$r_{kn}\to r_k$ as $n\to\infty$. It follows from $\tilde\mu_n\to
\tilde\mu$ that for sufficiently large $n$ the mark $(\sigma_{kn},
\pi_{kn})$ attached to $(x_{kn},y_{kn})$ is the same as the mark
$(\sigma_k,\pi_k)$ attached to $(x_k,y_k)$ for all $q'< k< q''$. The
same statement holds for $k=q'$ and $k=q''$ by the boundary conditions.
This implies that $\Phi_{0}(\tilde\mu_n, \sigma_n', \sigma_n'')\to
\Phi_{0}(\tilde\mu, \sigma', \sigma'')$ as $n\to\infty$.
\end{pf}
%
%le6.10 #&#
%
\begin{lemma}
We have $S_n(\kappa)\to S(\kappa)$ in distribution as $n\to\infty$.
%,where $S(\kappa)=\sum_{i=q'}^{q''-1} (\eps_i ^+ f^+(R_i)+\eps_i^-
%f^-(R_i))$.
\end{lemma}
\begin{pf}
By Proposition~\ref{propresnickreal} we have $\tilde\rho_n\to
\tilde\rho$ weakly on $\tilde\MMM$. The sum in (\ref
{eqtilderhontilderho}) can be taken from $1$ to $n-1$. Consequently,
$(\tilde\rho_n, \sgn\xi_0,\sgn\xi_n)$ converges weakly, as a
random element in $\KKK$, to $(\tilde\rho, \sigma',\sigma'')$,
where $\sigma'$ and $\sigma''$ are independent (and independent of
$\tilde\rho$) $\{-1,1\}$-valued random variables with the same
distribution as $\sgn\xi_0$. By Lemma~\ref{lemPsicontreal01} and
Proposition~\ref{propcontmapping} (which is applicable since $\P
[(\tilde\rho, \sigma',\sigma'')\in\KKK_0]=1$ for $\alpha\in
(0,1)$) we have that $\Phi_{0}(\tilde\rho_n, \sgn\xi_0,\sgn\xi_n)$
converges, as a random element in $\QQQ\times\QQQ$, to $\Phi_{0}(\tilde
\rho, \sigma',\sigma'')$ as $n\to\infty$. Taking the
integrals of $f^+$ and $f^-$ over the components of $\Phi_{0}(\tilde
\rho_n, \sgn\xi_0,\sgn\xi_n)$ and $\Phi_{0}(\tilde\rho, \sigma',\sigma
'')$, we arrive at the statement of the lemma.
\end{pf}

The proof of Theorem~\ref{theoreal} in the case $\alpha\in(0,1)$
can be completed as follows. Trivially, we have $S(\kappa)\to S$ a.s.
as $\kappa\downarrow0$. All the other assumptions of Lemma \ref
{lemgoodeventconvdistr} have been verified above. Applying
Lemma~\ref{lemgoodeventconvdistr}, we obtain $S_n\to S$ in
distribution as $n\to\infty$. The proof is complete.

%s6.4 #&#
\subsection{\texorpdfstring{Proof of Theorem \protect\ref{theorealalpha0}}{Proof of Theorem 1.14}}
It follows from the proof of Theorem~\ref{theocomplexalpha0} that on
the event $E_n$ defined as in Lemma~\ref{lem0En}, the number of real
roots of $G_n$ is the same as the number of real solution of the equation
%
%e70 #&#
%
\begin{equation}
\label{eqequivpolyforalpha0} \bigl(\xi_{\tau_n} z^{\tau_n}+
\xi_0 \bigr) \bigl(\xi_n z^{n-\tau_n}+
\xi_{\tau_n} \bigr)=0.
\end{equation}
The number of real solutions of (\ref{eqequivpolyforalpha0})
depends on whether the numbers $0,\tau_n,n$ are even or odd and on
whether the coefficients $\xi_0, \xi_{\tau_n},\xi_n$ are positive
or negative. It is not difficult to show that $(-1)^{\tau_n}$ and
$\sgn\xi_{\tau_n}$ become asymptotically independent and that $\P
[(-1)^{\tau_n}=1]\to1/2$ and $\P[\sgn\xi_{\tau_n}=1]\to c$ as
$n\to\infty$. Considering all possible cases leads to (\ref
{eqtheorealalpha0eq1}) and (\ref{eqtheorealalpha0eq2}).
%************************************************************
%************************************************************

%s7 #&#
\section{\texorpdfstring{Proofs of Theorems \protect\ref{theointensity} and \protect\ref{theoprobabtwoseg}}
{Proofs of Theorems 1.8 and 1.9}}
%s7.1 #&#
\subsection{\texorpdfstring{Proof of Theorem \protect\ref{theointensity}}{Proof of Theorem 1.8}}
Let $\rho$ be a Poisson point process with intensity $\nu
(du\,dv)=\alpha v^{-(\alpha+1)}\,du\,dv$ on $E=[0,1]\times(0,\infty)$,
where $\alpha\in(0,1)$. We are going to compute the expectation of
$L_{\alpha}$, the number of segments of the least concave majorant of
$\rho$. Denote by $\rho^2_{\ne}$ the set of all ordered pairs of
distinct atoms of the point process $\rho$. For $P_1,P_2\in E$
consider an indicator function $f_{\rho}(P_1,P_2)$ taking value $1$ if
and only if there are no points of the Poisson process $\rho$ lying
above the line passing through $P_1$ and $P_2$. Counting the first and
the last segments of the majorant of $\rho$ separately, we have $\E
L_{\alpha}=2+I_{\alpha}/2$, where
\[
I_{\alpha}=\E\biggl[\sum_{(P_1,P_2)\in\rho^2_{\ne}}f_{\rho
}(P_1,P_2)
\biggr].
\]
In the sequel we compute $I_{\alpha}$. Applying the Slyvnyack--Mecke
formula (see, e.g., \cite{schneiderweilbook}, Corollary 3.2.3), we obtain
\[
I_{\alpha}=\int_{E^2}\E
\bigl[f_{\rho}(P_1,P_2)\bigr]
\nu(dP_1)\nu(dP_2).
\]
%
%where $\theta$ is intensity measure of $\rho$.
Denoting $P_1=(x_1,y_1),P_2=(x_2,y_2)$, we have
\[
I_{\alpha}=\alpha^2\int_0^\infty
\int_0^\infty\int_0^1
\int_0^1\E\bigl[f_{\rho}(P_1,P_2)
\bigr]y_1^{-\alpha-1}y_2^{-\alpha-1}
\,dx_1\,dx_2\,dy_1\,dy_2.
\]
The probability of the event that there are no points of $\rho$ lying
above the line $P_1P_2$ is nonzero only if the line $P_1P_2$
intersects both vertical sides of the boundary of $E$. Therefore,
\[
I_{\alpha}=2\alpha^2\int_X\int
_{Y} \E\bigl[f_{\rho
}(P_1,P_2)
\bigr]y_1^{-\alpha-1}y_2^{-\alpha-1}
\,dy_1\,dy_2 \,dx_1\,dx_2,
\]
where $X=\{(x_1,x_2)\dvtx0<x_1<x_2<1\}$, and $Y=Y_{x_1,x_2}$ is a set
defined by
\[
Y= \bigl\{(y_1,y_2)\in(0,\infty)^2 \dvtx
y_1x_2-y_2x_1>0,
y_2-y_1+y_1x_2-y_2x_1>0
\bigr\}.
\]
Let us replace the variables $y_1,y_2$ by
\[
r=-\frac{y_2-y_1}{x_2-x_1},\qquad u=1+\frac{y_2-y_1}{y_1x_2-y_2x_1}.
\]
Then, $(y_1,y_2)\in Y$ if and only if $(r,u)\in(-\infty,0)\times
(1,\infty)$ or $(r,u)\in(0,\infty)\times(0,1)$. The inverse
transformation is given by
\[
y_1=r \biggl(\frac{1}{1-u}-x_1 \biggr),\qquad
y_2=r \biggl(\frac
{1}{1-u}-x_2 \biggr).
\]
%
%$$
%D=\left\{r,u \dvtx\frac{r}{1-u}>0, %-r+\frac{r}{1-u}>0\right\}=(-
%$$
The Jacobian determinant of the transformation $(r,u)\mapsto(y_1,y_2)$
is equal to
$
r(x_2-x_1)/(1-u)^2
$.
%Thus
% \,dx_1\,dx_2\,dr\,du.
Write $\tilde f_{\rho}(u,r)=f_{\rho}((x_1,y_1(u,r)),(x_2,y_2(u,r)))$.
By symmetry, we can consider only the case $r>0$, $u\in(0,1)$. Indeed,
considering the case $r>0$ means that we restrict ourselves to segments
of the majorant with positive slope. By a change of variables formula,
\begin{eqnarray*}
I_{\alpha}&=&4\alpha^2\int_0^\infty
\int_0^1\int_X \E
\bigl[\tilde f_{\rho}(u,r)\bigr]
\\
&&\hspace*{66.2pt}{}\times r^{-2\alpha-1} \biggl(\frac{1}{1-u}-x_1
\biggr)^{-\alpha
-1} \biggl(\frac{1}{1-u}-x_2
\biggr)^{-\alpha-1}\\
&&\hspace*{66.2pt}{}\times\frac
{x_2-x_1}{(1-u)^2} \,dx_1\,dx_2\,du\,dr.
\end{eqnarray*}
Further, by definition of the Poisson process,
%
%e71 #&#
%
\begin{eqnarray}
\label{0024} \E\bigl[\tilde f_{\rho}(u,r)\bigr] &=& \exp\biggl(-\int
_{\{(x,y)\in E \dvtx y\geq-rx+{r}/({1-u})\}} \alpha y^{-(\alpha+1)}\,
dy\,dx \biggr)
\nonumber\\
&=& \exp\biggl(-\int_{0}^1 \biggl(-rx+
\frac{r}{1-u} \biggr)^{-\alpha
}\,dx \biggr)
\\
%&=
&=& \exp\biggl(-\frac{r^{-\alpha}}{(1-\alpha)}
\frac{1-u^{1-\alpha
}}{(1-u)^{1-\alpha}} \biggr).
\nonumber
\end{eqnarray}
The integral $J:=\int_X(c-x_1)^\beta(c-x_2)^\beta(x_2-x_1)\,dx_1\,dx_2$,
where $c>1$, can be evaluated by writing $(x_2-x_1)=(c-x_1)-(c-x_2)$.
We obtain
%
%e72 #&#
%
\begin{equation}
\label{0317} J= \cases{ \dfrac{c^{2\beta+3}-(c-1)^{2\beta+3}-(2\beta
+3)c^{\beta
+1}(c-1)^{\beta+1}}{(\beta+1)(\beta+2)(2\beta+3)}, \vspace*{2pt}\cr
\hspace*{141pt}\quad \mbox{if $\beta\ne
-1,-3/2,-2$},
\vspace*{2pt}\cr
-4\ln
\biggl(\dfrac{c}{c-1} \biggr)+\dfrac4{\sqrt{c(c-1)}},
\qquad \mbox{if $\beta=-3/2$}.}
\end{equation}
%
%and for,
%$$
%$$
In the case $\alpha\neq1/2$, we apply (\ref{0024}) and (\ref{0317})
to obtain
\begin{eqnarray*}
I_{\alpha}&=&\frac{4\alpha}{(1-\alpha)(2\alpha-1)}\\
&&\hspace*{0pt}{}\times\int_0^\infty
\int_0^1r^{-2\alpha-1}\exp\biggl(-
\frac{r^{-\alpha}}{(1-\alpha
)}\frac{1-u^{1-\alpha}}{(1-u)^{1-\alpha}} \biggr)
\\
&&\hspace*{46.4pt}{}\times(1-u)^{2\alpha-3} \bigl[1-u^{1-2\alpha}-(1-2\alpha
)u^{-\alpha}(1-u) \bigr] \,du\,dr.
\end{eqnarray*}
In the case $\alpha=1/2$ we get, combining (\ref{0024}) with (\ref{0317}),
\[
I_{\alpha}=4\int_0^\infty\int
_0^1r^{-2}\exp
\biggl(-2r^{-1/2}\frac
{1-u^{1/2}}{(1-u)^{1/2}} \biggr) (1-u)^{-2}
\bigl[u^{-1/2}(1-u)+\ln u \bigr] \,du\,dr.
\]
Applying in both cases the formula $\int_0^\infty r^{-2\alpha
-1}e^{-cr^{-\alpha}} \,dr=(c^2\alpha)^{-1}$, we arrive at %at
%
%e73 #&#
%
\begin{equation}
\label{eqexpLalphaint} \E L_{\alpha}= \cases{\displaystyle  2+\frac{2(1-\alpha
)}{(2\alpha-1)}\int
_0^1 \frac{1-u^{1-2\alpha}-(1-2\alpha)u^{-\alpha
}(1-u)}{(1-u)(1-u^{1-\alpha})^{2}} \,du, \vspace*{3pt}\cr
\qquad\hspace*{142.5pt}\mbox{if $\alpha
\neq1/2$},
\vspace*{3pt}\cr
\displaystyle 2+\int_0^1\frac{u^{-1/2}(1-u)+\ln u}{(1-u)(1-u^{1/2})^2} \,du,
\qquad \mbox{if $\alpha=1/2$}.}
\end{equation}

%re7.1 #&#
%
\begin{remark}
The second line is just the limit of the first line as $\alpha\to
1/2$, so that $\E L_{\alpha}$ depends on $\alpha$ continuously. If
$\alpha=p/q\neq1/2$ is rational, then the substitution $v=u^{1/q}$
reduces the integral in (\ref{eqexpLalphaint}) to an integral of a
rational function which can be computed in closed form; see the table
in Section~\ref{subsecpropmaj}. Numerical computation suggests that
$\E L_{\alpha}$ is increasing in $\alpha\in(0,1)$.
\end{remark}
In the rest of the proof we compute the integral on the right-hand side
of (\ref{eqexpLalphaint}) in terms of the Barnes modular constant. Let
\[
K_{\alpha}=\int_0^1 \frac{1-u^{1-2\alpha}-(1-2\alpha)u^{-\alpha
}(1-u)}{(1-u)(1-u^{1-\alpha})^{2}}
\,du.
\]
%
%First we compute an auxiliary integral.
Write $\beta=1-\alpha$.
Recall that $\psi(z)=\Gamma'(z)/\Gamma(z)$ is the logarithmic
derivative of the Gamma function. Using the geometric series $\frac1
{1-u}=\sum_{n=0}^{\infty}u^n$ and the formula $\psi(z)=-\gamma
-\frac1 z+\sum_{n=1}^{\infty} (\frac1 n-\frac1 {z+n})$ (see
\cite{bateman}, Section 1.7) we obtain that for every $m>0$,
\begin{eqnarray*}
\int_{0}^1 u^{m\beta}\frac{1-u^{1-2\alpha}}{1-u}
\,du &=& \int_{0}^1 \sum
_{n=0}^{\infty} u^{n+m\beta}\bigl(1-u^{1-2\alpha}
\bigr) \,du
\\
&=& \sum_{n=1}^{\infty} \biggl(\frac1{n+m
\beta}-\frac{1}{n+(m+2)\beta
} \biggr) -\frac{1}{(m+2)\beta}
\\
%&=
&=& \psi\bigl((m+2)\beta\bigr)-\psi(m\beta)-\frac1 {m
\beta}.
\end{eqnarray*}
For $m=0$ the value of the integral is $\psi(2\beta)+\gamma$, where
$\gamma=-\psi(1)$ is the Euler--Mascheroni constant; see
\cite{bateman}, Section 1.7.2.
Using the expansion $\frac1{(1-u)^2}=\sum_{m=0}^{\infty}(m+1)u^m$ we
obtain that $K_{\alpha}=\lim_{N\to\infty}S_N$, where
\begin{eqnarray*}
S_N &=& \sum_{m=1}^{N}(m+1)
\biggl(\psi\bigl((m+2)\beta\bigr)-\psi(m\beta)-\frac
{1}{m\beta} \biggr)\\
&&{}-(N+1)
\frac{1-2\alpha}{1-\alpha}+\psi(2\beta)+\gamma
\\
&=& -2\sum_{m=1}^{N}\psi(m\beta)+(N+1)
\psi\bigl((N+2)\beta\bigr)+N\psi\bigl((N+1)\beta\bigr)-2N
\\
&&{}-\sum_{m=1}^N\frac1{m\beta}+\gamma-
\frac{1-2\alpha}{1-\alpha}.
\end{eqnarray*}
The second equality follows by an elementary transformation of the
telescopic sum.
%$$
%S_N=-2\sum_{m=1}^{N}\psi(m\beta)+(N+1)\log((N+2)\beta)+N\log%((N+1)
%$$
Using the asymptotic expansion $\psi(z)=\log z-\frac1 {2z}+o(\frac
1z)$ as $z\to\infty$, we obtain
\[
S_N=-2\sum_{m=1}^{N}\psi(m
\beta)+(2N+1)\log(\beta N)-2N-\frac1{\beta}\log N+1-\frac{\alpha\gamma
}{1-\alpha}+o(1).
\]
Comparing this with (\ref{eqbarnesmodularconst}) yields
\[
K_\alpha=1-2C(1-\alpha)+\frac{\log(1-\alpha)}{1-\alpha}-\frac
{\alpha\gamma}{1-\alpha}.
\]
The proof of Theorem~\ref{theointensity} is completed by inserting
this into (\ref{eqexpLalphaint}).
%we obtain
%$$
%I_{\alpha}=2\int_0^1\frac{u^{-1/2}(1-u)+\ln u}{(1-u^{1/2})^2(1-u)^2}
% \,du.
%$$
%Suppose now that $\alpha\ne1/2$. Applying \eqref{0024} and
%The intensity function is equal to
%$$
%(1-u)^{2\alpha-3}\left[1-u^{1-2\alpha}+(1-2\alpha)u^{-\alpha}(1-u)
%$$
%Again using \eqref{0039} we obtain
%Further, using the formula
%$
%$
%we obtain
%$$
%I_{\alpha}=\frac{4(1-\alpha)}{(1-2\alpha)}\int_0^1 %%(1-u)^{-1}(1-u^{1-
% \,du.
%$$

%s7.2 #&#
\subsection{\texorpdfstring{Proof of Theorem \protect\ref{theoprobabtwoseg}}{Proof of Theorem 1.9}}
We prove that $\P[L_{\alpha}=2]=1-\alpha$. For a point $P\in
E=[0,1]\times(0,\infty)$ let $g_{\rho}(P)$ be the indicator of the
following event: there are no atoms of $\rho$ above the lines joining
$P$ to the points $(0,0)$ and $(1,0)$. Then
\[
\P[L_{\alpha}=2]=\E\biggl[\sum_{P\in\supp\rho}
g_{\rho
}(P) \biggr].
\]
By the Slivnyak--Mecke formula \cite{schneiderweilbook}, Corollary 3.2.3,
%
%e74 #&#
%
\begin{equation}
\label{eqNalpha2sliv} \P[L_{\alpha}=2]=\int_E \E
\bigl[ g_{\rho}(P)\bigr] \nu(dP)= \alpha\int_{0}^1
\int_{0}^{\infty} \E\bigl[g_{\rho}(x,y)\bigr]
y^{-(\alpha+1)} \,dy\,dx.\hspace*{-20pt}
\end{equation}
The intensity of the Poisson process $\rho$ integrated over the set $\{
(u,v)\in E \dvtx u\in[0,x], v>yu/x\}$ is
\[
\int_{0}^x \int_{yu/x}^{\infty}
\alpha v^{-(\alpha+1)} \,du\,dv=\int_0^x
\biggl(\frac{yu}{x} \biggr)^{-\alpha} \,du=\frac{1}{1-\alpha}
xy^{-\alpha}.
\]
By symmetry, the intensity of $\rho$ integrated over the set $\{
(u,v)\in E \dvtx u\in[x,1],v>y(u-1)/(x-1)\}$ is $\frac{1}{1-\alpha}
(1-x)y^{-\alpha}$. It follows that
\[
\E\bigl[g_{\rho}(x,y)\bigr]=\exp\biggl(-\frac{1}{(1-\alpha)y^{\alpha}
} \biggr).
\]
Inserting this into (\ref{eqNalpha2sliv}) we obtain $\P[L_{\alpha
}=2]=1-\alpha$.

%suskaldyti doi

% imsref loaded by lrinkeviciute, 2012-08-21 10:24:35
% imsref loaded by lrinkeviciute, 2012-08-21 10:30:48
% imsref loaded by lrinkeviciute, 2012-08-21 10:39:16
% imsref loaded by lrinkeviciute, 2012-08-21 10:40:14
%

\printaddresses

\end{document}